\documentclass[11pt]{article}

\usepackage{amsmath, amssymb, amsthm, graphicx, caption, subcaption}
\usepackage{makeidx}  %%% standard INDEX
\usepackage[usenames,dvipsnames]{xcolor}  %% temporary, for coloured comments
\usepackage[bottom]{footmisc} 
\usepackage[colorlinks=true, linkcolor=blue, citecolor=green]{hyperref}  %% toggle colorlinks to hide hyperlinks in black text

\makeindex  

\graphicspath{{./Pictures/}}

\newtheorem{theorem}{Theorem}
\newtheorem{corollary}[theorem]{Corollary}
\newtheorem{prop}[theorem]{Proposition}
\newtheorem{lem}[theorem]{Lemma}

\newtheorem{lemma}[theorem]{Lemma}

\newtheorem{observation}[theorem]{Observation}

\theoremstyle{definition}

\newtheorem*{definition}{Definition}

\newtheorem{case}{Case}

\newcommand{\startingcases}{\setcounter{case}{0}}

\def\comp{\text{comp}}

\def\cut{\setminus}
\def\contract{/}

\let\emptyset\varnothing

\title{Forbidden minors for graphs with no first obstruction to parametric Feynman integration}
\author{Samson Black, Iain Crump, Matt DeVos, and Karen Yeats}

\begin{document}
\maketitle
\begin{abstract}
  We give a characterization of 3-connected graphs which are planar and forbid cube, octahedron, and $H$ minors, where $H$ is the graph which is one $\Delta-Y$ away from each of the cube and the octahedron.  Next we say a graph is \emph{Feynman 5-split} if no choice of edge ordering gives an obstruction to parametric Feynman integration at the fifth step.  The 3-connected Feynman 5-split graphs turn out to be precisely those characterized above.  Finally we derive the full list of forbidden minors for Feynman 5-split graphs of any connectivity.
\end{abstract}

\tableofcontents

\section{Introduction}

The Robertson-Seymour theorem \cite{RSXX} tells us that any minor closed graph property is defined by a finite set of forbidden minors.  The set of forbidden minors itself can vary from the sublime, such as Wagner's theorem, to the ridiculous, like the sixty-eight billion (and counting) forbidden minors for $\Delta-Y$ reducibility \cite{YuYDY}.  Middle ground includes cases like \cite{DDpathwidth} where the 3-connected situation is simple while the general case is more intricate.

In this paper we are interested in a minor closed graph property, called \emph{Feynman 5-splitting}, which originates from a physics residue calculation.  The property of interest is defined in Section \ref{sec splitting} in its original physics manifestation, and matroidally using Edmonds’ matroid intersection theorem in subsection \ref{matroid splitting}.   

Briefly, Feynman graphs in quantum field theory encode integrals which describe particle interactions.  Francis Brown \cite{Brbig} developed an algebro-geometric algorithm to integrate certain scalar Feynman integrals one edge of the graph at a time.  The denominators after each step of the integration are key to the algorithm and also can be interpreted combinatorially as certain polynomials defined from the original graph.  A graph which is not Feynman 5-split is a graph which, for at least one order of its edges, has an obstruction to continuing the algorithm after the fifth step.  

\medskip

Our first result in the classification of the excluded minors for Feynman 5-splitting is a purely graph theoretic structure theorem showing that a simple 3-connected graph $G$ has a certain width property, closely related to pathwidth 3, if and only if 
$G$ does not have one of a family of five minors.  To state this precisely we require some further terminology.  Let $C$ and $O$ denote the cube and octahedron graphs respectively, let $H$ denote the graph depicted in Figure \ref{firsth}, and define $\mathcal{F}_0 = \{ K_{3,3}, K_5, C, H, O \}$.  

\begin{figure}[h]
  \centering
    \includegraphics[scale=1.0]{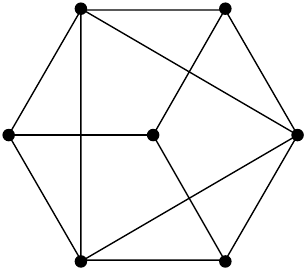}
  \caption{$H$}
\label{firsth}
\end{figure}

For any set of edges $A \subseteq E(G)$ we let $G_A$ denote the subgraph of $G$ induced by $A$.  A \emph{separation} of $G$ consists of a pair $(A,B)$ of subsets of $E(G)$ for which $A \cap B = \emptyset$ and $A \cup B = E(G)$.  We say that this is a separation \emph{on} $V(G_A) \cap V(G_B)$, we call $| V(G_A) \cap V(G_B) |$ the \emph{order} of the separation, and write $\partial(A)=\partial(B)=V(G_A) \cap V(G_B)$.  If the order of $(A,B)$ is at most $k$, we call $(A,B)$ a $k$-separation.  Finally, a separation $(A,B)$ is \emph{proper} if both $V(G_A) \setminus V(G_B)$ and $V(G_B) \setminus V(G_A)$ are nonempty.  

\begin{definition}[\cite{Brbig} section 1.4\footnote{Brown uses the term ``vertex width'' instead of ``width''.  This notion is also known as ``linear-width'' by Thilikos \cite{Twidth}}]
Let $G$ be a graph with $n$ edges.  The \emph{width} of an ordering $e_1, e_2, \ldots, e_n$ of $E(G)$ is the maximum order of 
a separation of the form $( \{e_1, \ldots, e_{\ell}\}, \{e_{\ell+1}, \ldots, e_n \} )$.  The \emph{width} of $G$ is the minimum width among 
all edge orders of $G$.
\end{definition}

The notion of width will be explored in detail in Section \ref{sec matroids}, however let us note here that it is closely related to the well-known concept of path width.  We may state the structure theorem as follows.

\begin{theorem} \label{decomp} A simple 3-connected graph has width at most $3$ if and only if it has no minor in $\mathcal{F}_0$.
\end{theorem}

The proof of this theorem appears in Section \ref{seccharac}. Related characterizations in the literature include cube-free graphs \cite{Mcube}, octahedron-free graphs \cite{MR3090713} and planar cube-free graphs \cite{PScube}.  In the other direction, Thilikos \cite{Twidth} gave the list of forbidden minors for width at most 2.

Now we turn our attention back to the notion of Feynman 5-splitting. %, which we will henceforth abbreviate as \emph{splitting} for the sake of brevity.  
As with \cite{DDpathwidth}, our forbidden minor result breaks into the 3-connected simple case and the non-3-connected case.  
%Split graphs (i.e. graphs which are Feyman 5-splitting) 
Feynman 5-split graphs which are at least 3-connected turn out to again be those 3-connected graphs which forbid $\mathcal{F}_0$ (Theorem \ref{heyguesswhatsplits}). 
The 2-connected case is more intricate.  First we must observe that certain small 2-cuts are functionally the same for the purposes of 
Feynman 5-splitting 
(Section \ref{wands}).  This then gives 34 more functionally distinct forbidden minors to complete our characterization  (Lemma \ref{fminor}).

\section{3-connected graphs with no $\mathcal{F}_0$ minor}\label{seccharac}

In this section we prove Theorem \ref{decomp}.  This argument will require the notion of a rooted minor.  Let $G$ and $R$ be loopless graphs.  As usual, we say that $G$ has $R$ as a \emph{minor} if for every $v \in V(R)$ there exists a set of vertices $X_v \subseteq V(G)$ satisfying the following:
\begin{itemize}
\item $X_v \cap X_w = \emptyset$ whenever $v,w \in V(R)$ and $v \neq w$.  
\item The subgraph of $G$ induced by $X_v$ is connected for every $v \in V(R)$.
\item For all $v,w \in V(R)$ with $v \neq w$, the number of edges in $G$ between $X_v$ and $X_w$ is at least the number of edges between $v$ and $w$ in $R$.
\end{itemize}
If $U \subseteq V(R)$ and $T \subseteq V(G)$ are the same size and in addition $|T \cap X_{u_i}| = 1$ for every $1 \le i \le n$, then $(G,T)$ has a \emph{rooted} $(R,U)$-minor.  If $U = \{u_1, \ldots, u_n \} \subseteq V(R)$ and $T = \{t_1, \ldots, t_n\} \subseteq V(G)$ and in addition 
$t_i \in X_{u_i}$ for every $1 \le i \le n$, then $(G;  t_1, \ldots, t_n)$ has a \emph{rooted} $(R; u_1, \ldots, u_n)$ minor.  In these cases we refer 
to the vertices of both $T$ and $U$ as \emph{roots}.

\begin{figure}[h]
  \centering
    \includegraphics[scale=0.6]{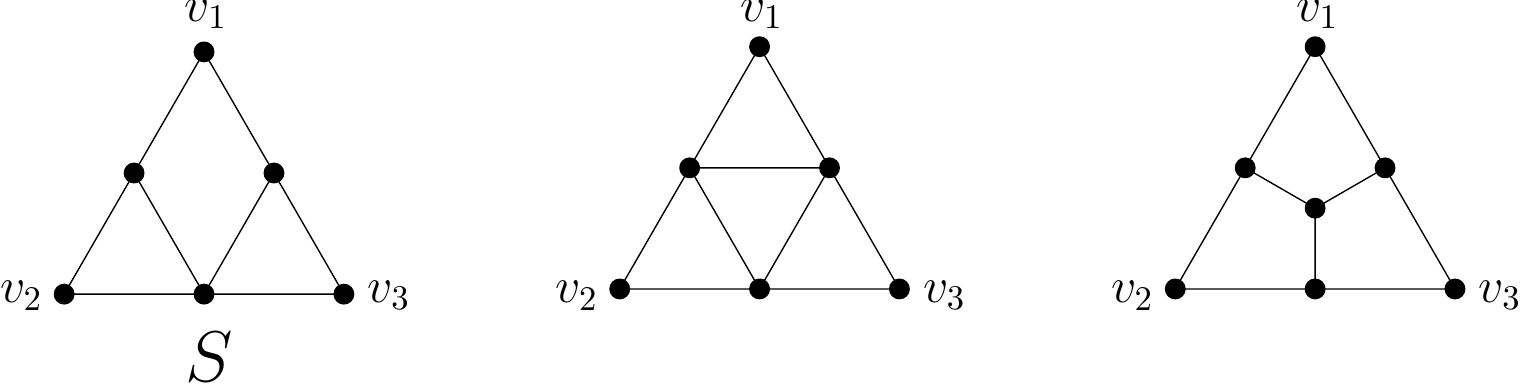}
  \caption{}
\label{s_plus}
\end{figure}

We define $S$ to be the first graph depicted in Figure \ref{s_plus} and consider the vertices $\{v_1,v_2,v_3\}$ to be roots.  

\begin{lem} \label{fan2} Let $G$ be a simple 3-connected graph with no minor in $\mathcal{F}_0$ and let $(A,B)$ be a proper 3-separation on $X = \{x_1,x_2,x_3\}$.  If $X$ is independent in $G_B$ and $(G_B; x_1, x_2, x_3)$ has a rooted $(S; v_1, v_2, v_3)$-minor, then $G_A \setminus x_1$ is a path from $x_2$ to $x_3$.
\end{lem}

\begin{proof} If $G_A \setminus x_1$ contains a cycle $D$, then by the 3-connectivity of $G$ we may choose three vertex disjoint paths in $G_A$ starting at $x_1,x_2,x_3$ and terminating in $D$ (the paths starting at $x_2$ and $x_3$ may be trivial).  It then follows that $G$ contains $H$ as a minor, which is a contradiction.  

Since $G$ is 3-connected, the graph $G_A \setminus x$ is connected and so must be a tree.  If $u$ is a leaf of this tree which is not one of $x_2,x_3$, 
then $u$ has degree at most 2 in $G$, a contradiction.  It follows that $G_A \setminus x_1$ is a path from $x_2$ to $x_3$ as desired.  
\end{proof}

\begin{definition}
Consider a graph $G$ with a separation $(A,B)$ on $X$.  We say that $A$ is \emph{rich} if for every $x \in X$ there exists a cycle 
$D \subseteq G$ with $x \not\in V(D)$ so that $|E(D) \setminus A| \le 1$, and similarly for $B$.  If $A$ ($B$) is not rich, we call it \emph{poor}.
\end{definition}

Note that in the previous lemma we may add the conclusion that $A$ is poor since there does not exist  a cycle in $G \setminus x_1$ containing 
at most one edge of $B$. 

\begin{lem}
\label{getsrich}
Let $G$ be a 3-connected graph with no $\mathcal{F}_0$ minor which has a proper 3-separation $(A,B)$ on $X$.  If
$X$ is independent in $G_B$ and $\mathit{deg}_{G_B}(x) \ge 2$ for every $x \in X$, then $A$ is poor, and there exists a 
3-separation $(A',B')$ for which $A'$ and $B'$ are rich.
\end{lem}

\begin{proof} Let $X = \{x_1,x_2,x_3\}$ and consider an embedding of $G$ in the sphere.   For $1 \le i < j \le 3$ let $P_{ij}$ be a non-trivial path in $G_B$ from $x_i$ to $x_j$ which forms part of the boundary of a face.  Note that since $G$ is 3-connected, these paths do not depend on the choice of embedding.  In addition, our assumptions imply that each $P_{ij}$ has at least one interior vertex and are internally disjoint.  By restricting our embedding of $G$ we may obtain an embedding of $G_B$ in a disc with the cycle $D = P_{12} \cup P_{23} \cup P_{13}$ on the boundary of the disc.  Define a \emph{bridge} to be a subgraph of $G_B$ which either consists of a single edge $e \in E(G_B) \setminus E(D)$ for which both ends of $e$ are in $V(D)$, or consists of a component $G'$ of $G_B \setminus V(D)$ together with all edges joining a vertex of  $G'$ and a vertex of $D$ (together with any endpoint of such an edge).  An \emph{attachment} of a bridge $K$ is a vertex in $V(K) \cap V(D)$.  Note that by the embedding of $G_B$ and the 3-connectivity of $G$, for every bridge $K$, and every path $P_{ij}$, the attachments of $K$ which lie in $V(P_{ij})$ must form an interval of this path, and $K$ must have an attachment outside of this path.  

If there is a bridge with attachments in the interiors of $P_{12}$,  $P_{23}$ and $P_{13}$, then $(G_B, \{x_1,x_2,x_3\})$ has the third graph from 
Figure \ref{s_plus} together with $\{v_1,v_2,v_3\}$ as a rooted minor.  It then follows that $G$ has $C$ as a minor, which is a contradiction.
Next suppose that there exists a bridge $K$ with $x_1$ and a vertex in the interior of $P_{23}$ as attachments.  We may assume (without loss) that there are no attachments of $K$ in the interior of $P_{12}$.  Let $u$ be the attachment of $K$ on $P_{23}$ which is closest along this path to $x_2$, and 
note that $u \neq x_2$ (otherwise this would force all of $V(P_{12})$ to be attachments of $K$).  It now follows that $G_B$ has a 2-separation on 
$\{x_1,u\}$.  Furthermore, by planarity, every bridge with an attachment in the interior of $P_{12}$ ($P_{13}$) also has an attachment in the interior of $P_{23}$.  Therefore $(G_B,X)$ has a rooted $(S,\{v_1,v_2,v_3\})$ minor.  Now Lemma \ref{fan2} implies that $G_A \setminus x_1$ is a path, so $G_A$ is poor.  Furthermore, for any interior vertex $w$ of this path we have that $w$ and $x_1$ are adjacent (by 3-connectivity) and thus there is a 3-separation $(A',B')$ of $G$ on $\{w,x_1,u\}$ for which both $A'$ and $B'$ are rich.  Argue similarly for any bridge from $x_2$ or $x_3$ to the interior of the opposite side.

We may now assume that there does not exist a bridge with attachments as considered in the previous paragraph.  It follows from this that every bridge has attachments in the interior of exactly two of the three paths $P_{12}$, $P_{13}$, $P_{23}$.  If for each pair of these paths there is a bridge which attaches to both interiors, then $(G_B, \{x_1,x_2,x_3\})$ has the second graph from figure \ref{s_plus} together with $\{v_1,v_2,v_3\}$ as a rooted minor and it follows that $G$ contains $H$ as a minor, which is contradictory.  So, we may assume (without loss) that there are no bridges with attachments in the interior of both $P_{12}$ and $P_{13}$.  It follows from this that $x_1$ is not an attachment of any bridge.  Now let $u$ be the vertex on $P_{23}$ which is closest to $x_2$ along this path with the property that $u$ is an attachment of a bridge which also has an attachment in the interior of $P_{13}$.  Now $G_B$ has a 2-separation on $\{x_1,u\}$ and $(G_B,X)$ has a rooted $(S,\{v_1,v_2,v_3\})$ minor.  As before, Lemma \ref{fan2} implies that $G_A \setminus x_1$ is a path, so $G_A$ is poor, and for any interior vertex $w$ of this path, we have a 3-separation $(A',B')$ of $G$ on $\{w,x_1,u\}$ for 
which both $A'$ and $B'$ are rich.
\end{proof}

The proof of Theorem \ref{decomp} also requires the following classical theorem on graph minors. 

\begin{theorem}[Halin, Jung \cite{HJ}]
\label{octk5}
Every simple graph with minimum degree at least four contains either $K_5$ or $O$ as a minor.
\end{theorem}

\begin{proof}[Proof of Theorem \ref{decomp}.] \mbox{}
Let $G = (V,E)$ be a simple 3-connected graph with no minor in $\mathcal{F}_0$.
By Theorem \ref{octk5} we may choose a vertex $v$ of degree 3.  We now form a sequence $e_1, e_2, \ldots$ of edges by choosing $e_1,e_2,e_3$ to be the three edges incident with $v$ and then greedily extending the list by a new edge $e_k$ if there exists such an edge with the property that $( \{e_1, \ldots, e_k\} , E \setminus \{e_1, \ldots, e_k \}  )$ is a 3-separation.  If this procedure exhausts all of the edges, then we are done.  Otherwise, at the point where we get stuck, we have a 3-separation $( \{e_1, \ldots, e_{\ell} \} , E \setminus \{e_1, \ldots, e_{\ell} \}  )$ on $X$ with the property that no edge in $E \setminus \{e_1, \ldots, e_{\ell} \}$ has both ends in $X$, and every vertex in $X$ has degree at least two in $G_{ E \setminus \{e_1, \ldots, e_{\ell} \} }$.  So we may apply Lemma \ref{getsrich} to choose a 3-separation $(A',B')$ of $G$ for which $A'$ and $B'$ are both rich.  Now we form a sequence of edges $b_1, b_2, \ldots $ of $B'$ in a greedy fashion by repeatedly choosing an edge $b_{i}$ 
so that $(A' \cup \{b_1, \ldots, b_i \}, B' \setminus \{b_1, \ldots, b_i \} )$ is a 3-separation.  If this procedure terminates before all edges of $B'$ have been used, say after choosing $b_i$, then we have a 3-separation $(A' \cup \{b_1, \ldots, b_i\}, B' \setminus \{b_1, \ldots, b_i\} )$ satisfying the hypothesis of Lemma \ref{getsrich} and this implies that $A' \cup \{b_1, \ldots, b_i\}$ is poor, which is contradictory to $A'$ being rich.  Therefore, this process terminates with a sequence $b_1, \ldots, b_m$ containing all edges of $B'$.  Similarly, we may greedily sequence the edges of $A'$ as $a_1, a_2, \ldots a_{\ell}$ so that $(A' \setminus \{a_1, \ldots, a_j\} , B' \cup \{a_1, \ldots, a_j \})$ is a 3-separation for every $1 \le j \le \ell$.  Now the edge sequence 
$a_{\ell}, \ldots, a_1, b_1, \ldots, b_m$ has width 3.

For the other direction, note that the graphs $K_{3,3}$, $K_5$, $C$, and $O$ do not contain a 3-separation $(A,B)$ with $|A|, |B| \ge 4$,  and the graph $H$ has no such 3-separation $(A,B)$ with $|A|, |B| \ge 5$.  It follows that all of these graphs have width at least 4.
\end{proof}

\section{Splitting}\label{sec splitting}

In this section we introduce the notion of Feynman 5-splitting, which serves as the central focus of our investigation.  
Given a not necessarily simple graph $G$ assign to each edge $e$ a variable $x_e$.  Then the \emph{Kirchhoff} polynomial of $G$ is
\[
\Psi_G = \sum_{\substack{T \text{ spanning}\\\text{tree of } G}} \left( \prod_{e\not\in T}x_e \right)
\]
%Note that this definition is dual to the more common definition of the Kirchhoff polynomial.  
The \emph{Feynman period} of $G$ is
\[
\int_{0}^\infty\cdots \int_0^\infty \frac{\prod_e dx_e}{\Psi_G^2}\delta(\sum_{e}x_e-1)
\]
where $\delta$ is the Dirac delta.
This is only a part of the full Feynman integral of a graph, which in general would involve external momenta, masses, and for non-scalar Feynman diagrams, further terms in the numerator incorporating the tensor structure of the diagram.  However, it is an important part physically. The full Feynman integral typically diverges, and with appropriate regularization this part is the coefficient of logarithmic growth at infinity, hence is a kind of residue \cite{numbers}.  Furthermore for primitive divergent graphs, defined below, this part gives the most complicated contributions to the $\beta$-function of the theory \cite{Sphi4} and is invariant under a wide variety of renormalization schemes since no choices have to be made for subdivergences.  It is also mathematically a very interesting object both algebro-geometrically; see for example \cite{AMdet, bek, BrS}, and more combinatorially \cite{BrSY, BrY, SFq}.

Let $\ell(G)$ be the cyclotomic number of $G$, the minimum number of edges needed to be removed to create an acyclic graph, and let $e(G)$ be the number of edges of $G$.  The Feynman period converges provided $\ell(G) = 2e(G)$ and $\ell(\gamma) < 2e(\gamma)$ for any proper subset of edges $\gamma$ of $G$. Graphs that meet this requirement are known as \emph{primitive divergent}.  Calculating the values of Feynman periods is quite difficult.  Oliver Schnetz \cite{Sphi4} following older work of David Broadhurst and Dirk Kreimer \cite{bkphi4} calculated as many as possible with current techniques.

There is also a matrix-theoretic point of view on the Kirchhoff polynomials which is very useful.
Following Brown \cite{Brbig} we will use an exploded Laplacian which is well suited to our form of the Kirchhoff polynomial.

Given an undirected graph $G$, choose an arbitrary orientation for the edges and let $\xi_G$ be the $|E(G)| \times |V(G)|$ incidence matrix for this directed graph.  
%Specifically, $\xi_G$ is a $|E(G)| \times |V(G)|$ matrix such that $$ (\xi_G)_{e,v} =\begin{cases} 1 & \text{if } v \text{ is the source vertex of edge } e, \\ -1 & \text{if } v \text{ is the target vertex of edge } e, \\ 0 & \text{otherwise.} \end{cases} $$
%
% This is the transpose of the more standard definition of the incidence matrix. 
Let $A$ be the diagonal matrix with entries $x_e$ for $e \in E(G)$ with the same choice of order as in the construction of $\xi_G$.  
%Then define $\widetilde{M}_G$ to be the block matrix constructed as follows; \begin{equation*} \widetilde{M}_G = \left[\begin{array}{c|c} A & \xi_G \\ \hline -\xi^T_G & \bf{0}\end{array}\right].\end{equation*} The first $|E(G)|$ rows and columns are indexed by the edges of $G$, and the remaining $|V(G)|$ rows and columns are indexed by the set of vertices of G, in some arbitrary order. 
Let $\widehat{\xi}_G$ be a submatrix of the incidence matrix $\xi_G$ obtained by deleting an arbitrary column. Define the matrix $M_G$ as, 
\[
 M_G = \left[\begin{array}{c|c}A & \widehat{\xi}_G \\ \hline -\widehat{\xi}^T_G & \bf{0}\end{array}\right].
\] 
The first $|E(G)|$ rows and columns are indexed by the edges of $G$, and the remaining $|V(G)|-1$ rows and columns are indexed by the set of vertices of $G$ other than the one removed in the construction of $\widehat{\xi}$.

Then the matrix-tree theorem implies
\[
\Psi_G = \det(M_G)
\] and hence that the determinant of $M_G$ does not depend on the orderings, orientations, or choice of removed vertex (see Proposition 3.7 of \cite{Ycov} for a concise derivation of this from the standard form of the matrix-tree theorem). 

%
%\begin{theorem} \label{determinant} For an arbitrary graph $G$, the Kirchhoff polynomial $\Psi_G = \det(M_G)$. Specifically, the determinant does not depend on the edge orientation of $G$, nor the column deleted from $\xi_G$. \end{theorem}

In \cite{Brbig} Francis Brown gave an algorithm to compute some Feynman periods one edge variable at a time.  The key to the algorithm is the structure of the denominator at each step; if the denominator at a given step factors into two terms which are each linear in a common variable, then the algorithm can proceed to the next step.  The numerators at each stage are explicit polylogarithms.  The obstruction to this algorithm is the denominator not factoring.  The first time when this can occur is at the fifth step.

\begin{definition} Let $\Psi_G$ be the Kirchoff polynomial for a graph $G$ and $M_G$ the matrix used above. Let $I,J,K \subseteq E(G)$. Let $M_G(I,J)_K$ be the matrix obtained by deleting rows indexed by the edges in $I$, columns indexed by the edges in $J$, and setting $\alpha_e = 0$ for all $e \in K$. 
If $|I| = |J|$ define the \emph{Dodgson polynomial}
\[
\Psi^{I,J}_K = \det(M_G(I,J)_K).
\] 
If $K = \emptyset$, write this as $\Psi^{I,J}$. \end{definition}

It is demonstrated in \cite{Brbig} that Dodgson polynomials are well-defined up to overall sign.  For any given calculation, by fixing a choice of matrix $M_G$ the relative signs between Dodgson polynomials are also well-defined.

\begin{definition} For a graph $G$, a \emph{5-configuration} is a set $S \subseteq E(G)$ such that $|S| = 5$. \end{definition}

\begin{definition} Let $G$ be a graph and $S = e_1, ..., e_5$ a 5-configuration of $G$. The \emph{five-invariant} is the polynomial 
\[
^5\Psi(e_1,e_2,e_3,e_4,e_5) = \pm (\Psi^{e_1e_2,e_3e_4}_{e_5} \Psi^{e_1e_3e_5,e_2e_4e_5} - \Psi^{e_1e_3,e_2e_4}_{e_5} \Psi^{e_1e_2e_5,e_3e_4e_5}).
\]
 \end{definition}

The 5-invariant is the denominator at the fifth stage of integration.  It was first implicitly found in \cite{bek}, equation (8.13).

%The following is Lemma 87 in \cite{Brbig}.

\begin{lem}[\cite{Brbig} Lemma 87] \label{reindex} Reordering the edges in a five-invariant may at most change the sign of the polynomial. \end{lem}

The Dodgsons which appear in the definition of the 5-invariant must have  $|I \cup J \cup K| =5$, $|I \cap K| = |J \cap K| = 0$, and either $|I| = |J| = 2$, $|K| = 1$, and $|I \cap J| = 0$ or $|I| = |J| = 3$, $|K| = 0$, and $|I \cap J| = 1$. Trivially, $\Psi^{I,J}_K =\Psi^{J,I}_K$. Thus
for a fixed 5-configuration $S$, there are thirty generically distinct Dodgsons of these forms associated with $S$. %For notational convenience we will write $\Psi^{e_1e_2,e_3e_4}_{e_5}$ in place of $\Psi^{ \{e_1,e_2\}, \{e_3,e_4\}}_{\{e_5\}}$, as an example.

%\begin{remnum} \label{plusminus} Recall from the definition that a Dodgson of a graph $G$ is well-defined up to overall sign. Hence, it is necessary to fix a matrix $M_G$ in calculating the five-invariant. Specifically, the choice of matrix will affect each Dodgson in a predictable manner, and calculating the same five-invariant using a different matrix will only affect the overall sign. \end{remnum}

\begin{definition} Let $G$ be a graph and fix a 5-configuration $S$ in $G$. We say that $S$ \emph{splits} if at least one of the 30 Dodgson polynomials associated to $S$ is identically zero. 
%If all Dodgson polynomials for a 5-configuration are non-zero, we say that $S$ is a \emph{non-splitting} 5-configuration. 
If $S$ splits for every possible 5-configuration $S \subseteq E(G)$, we say that $G$ itself is Feynman 5-split which we will henceforth abbreviate as \emph{split} for the sake of brevity. %If there exists at least one 5-configuration $S$ such that $S$ does not split, we say that $G$ is a \emph{non-splitting} graph. 
\end{definition}

If a 5-configuration $\{e_1,e_2,e_3,e_4,e_5\}$ splits, then, by Lemma \ref{reindex}, it is possible to permute the indices such that $^5\Psi(e_1,...,e_5) = \pm \Psi^{e_1e_2,e_3e_4}_{e_5} \Psi^{e_1e_3e_5,e_2e_4e_5}$. As each Dodgson is, by construction, linear in each variable, this five-invariant can be factored into a product of polynomials that is linear in each variable.  Therefore, splitting is a way of avoiding the obstruction to the integration algorithm at the fifth step.  It is a theoretically nice way to avoid the obstruction since the factorization comes for combinatorial reasons, and it is also in practice the sort of factorization which occurs.  

On the other hand, splitting is a very strong condition as it requires not just that there is some way to avoid the obstruction but that every choice of 5 edges avoids the obstruction.  Furthermore Theorems \ref{decomp} and \ref{heyguesswhatsplits} tell us that split graphs have width 3. 
A result of Brown (\cite{Brbig} Theorems 1 and 2) is that for width 3 graphs there is at least one edge ordering so that the algorithm can continue until all edges have been integrated.  Thus splitting requires that all ways of starting the algorithm are well behaved but as a consequence gives that there is some ordering with no obstructions at any point.

The following proposition provides a more graph-theoretic method of calculating Dodgson polynomials that will prove useful.  

\begin{prop} \label{calcdodgson} For a graph $G$ and $I,J,K \subseteq E(G)$ such that $|I| = |J|$, the Dodgson polynomial $$\Psi^{I,J}_K = \sum_{T \subseteq E(G)} \left( \pm \prod_{e \notin T \cup S} x_e \right)$$ where $S = I \cup J \cup K$ and the sum is over all edge sets which induce spanning trees in both graph minors $G \cut I \contract ((J \cup K) - I)$ and $G \cut J \contract ((I \cup K) - J)$. \end{prop}

In particular, the above proposition immediately implies the following useful corollary.

\begin{corollary}\label{split-tree-cor} 
Let $G$ be a graph and let $S \subseteq E(G)$ be a 5-configuration.  Then $S$ splits if and only if is has a partition $\{ \{e\}, S_1 , S_2\}$ with $|S_1| = 2 = |S_2|$ so that one of the graphs $G \setminus e$ or $G/e$ does not have a set of edges $T$ for which both $T \cup S_1$ and $T \cup S_2$ are spanning trees.
\end{corollary}

Key for us is the observation that being split is a minor closed property. Hence, by the Robertson-Seymour Theorem, our goal is to find the forbidden minors for splitting.

\begin{prop} Splitting is a minor-closed property. \end{prop}

\begin{proof} Suppose a graph $G$ splits. For any 5-configuration $S$ and edge $e \notin S$, the deletion or contraction of $e$ cannot create more spanning trees. It follows from Corollary \ref{split-tree-cor} that $G\cut e$ and $G \contract e$ must also split. \end{proof}

The following proposition is another key property of splitting that will be useful.

\begin{prop}\label{closedunderdual} Let $G$ be a planar graph and $G^*$ its planar dual. Then $G$ splits if and only if $G^*$ splits. \end{prop}

\begin{proof}

This follows immediately from Corollary \ref{split-tree-cor}, using the edges dual to those not in the tree in $G$ to create a tree in $G^*$, and from the fact that edge deletion and contraction are dual operations. It follows immediately from this that if a graph $G$ is minor-minimal non-splitting, then $G^*$ must be also.
\end{proof}

The following calculations using the 5-configurations indicated in Figure \ref{nonsplitfives} show that the graphs in
$\mathcal{F}_0$ are non-split.  In the next section, following Corollary \ref{whatsplitmeans}, we will provide an alternate proof of this fact.

\begin{figure}[h]
  \centering
    \includegraphics[scale=0.95]{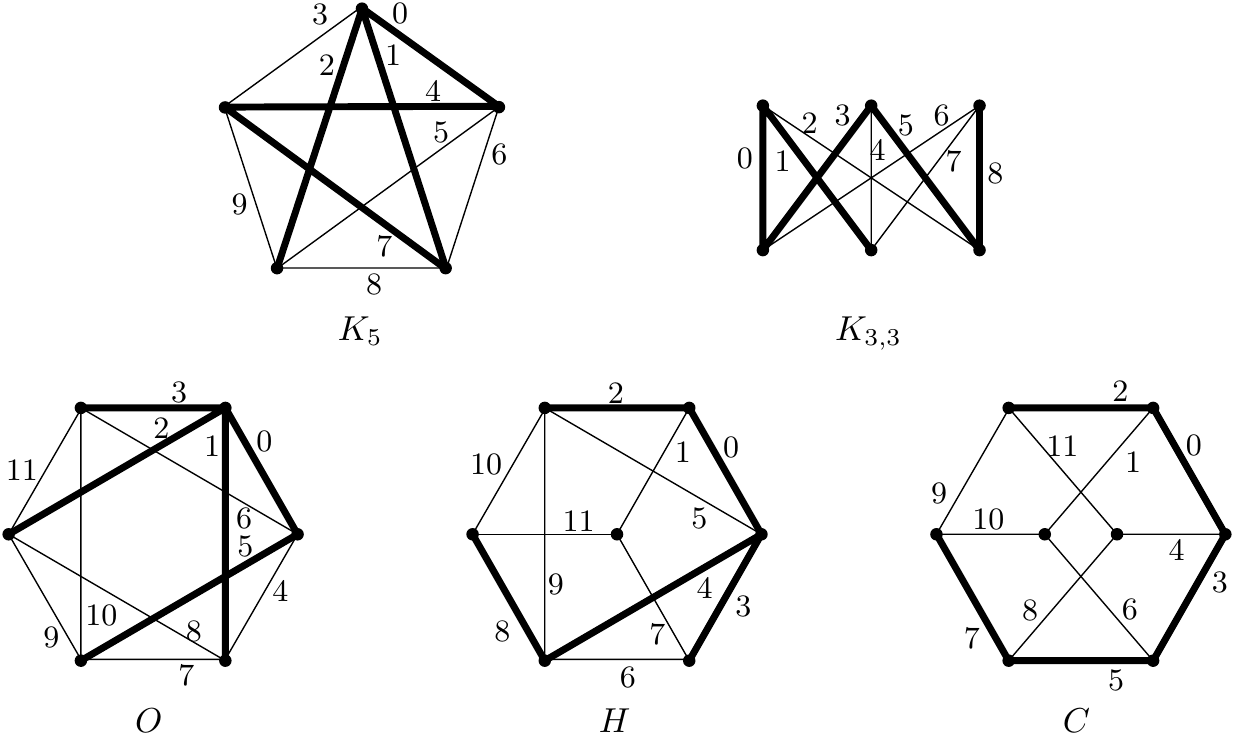}
  \caption{Non-splitting five configurations.}
  \label{nonsplitfives}
\end{figure}

\begin{align*} ^5\Psi_{K_5} (e_0,e_1,e_2,e_6,e_8) = & x_7 x_5 x_4 x_3 (x_3 x_5 x_7 + x_3 x_4 x_9 \\
&+ x_3 x_5 x_9 + x_3 x_7 x_9 + x_4 x_9^2)\end{align*}
\begin{align*} ^5\Psi_{K_{3,3}} (e_0,e_1,e_3,e_5,e_8) = & x_2 x_4 x_6 + x_2 x_4 x_7 + x_2 x_6 x_7  - x_4 x_6 x_7 + x_2 x_7^2 \end{align*}
\begin{align*} ^5\Psi_{O} (e_0,e_1,e_2,e_3,e_5) = 	&x_{11} x_9 x_8 x_6 x_4 ( -x_6 x_7^2 x_9 - x_6 x_7^2 x_{10} \\
&+ x_4 x_7 x_9 x_{10} - x_6 x_7 x_9 x_{10} + x_4 x_7 x_{10}^2 \\
&+ x_4 x_8 x_{10}^2 + x_7 x_8 x_{10}^2 + x_4 x_9 x_{10}^2 \\
&- x_6 x_7^2 x_{11} - x_7^2 x_{10} x_{11} ) \end{align*}
\begin{align*} ^5\Psi_{H} (e_0,e_1,e_3,e_4,e_8) = &x_6  x_5  ( x_2 x_5 x_6 x_9 x_{10} + x_2 x_5 x_7 x_9 x_{10} \\ 
&+ x_2 x_5 x_6 x_{10}^2 + x_2 x_5 x_7 x_{10}^2 + x_2 x_6 x_9 x_{10}^2 \\
&+ x_2 x_7 x_9 x_{10}^2 - x_2 x_5 x_9^2 x_{11} -x_2 x_7 x_9^2 x_{11} \\
&- x_2 x_5 x_9 x_{10} x_{11} + x_2 x_6 x_9 x_{10} x_11 - x_2 x_9^2 x_{10} x_{11} \\
&- x_7 x_9^2 x_{10} x_{11} - x_2 x_9^2 x_{11}^2 - x_7 x_9^2 x_{11}^2 ) \end{align*}
\begin{align*} ^5\Psi_{C} (e_0,e_1,e_3,e_5,e_7) = &-x_2 x_4 x_6 x_8 x_9 - x_2 x_4 x_6 x_9^2 - x_4 x_6 x_8 x_9^2 \\
&- x_4 x_6 x_8 x_9 x_{10} -x_2 x_4 x_6 x_9 x_{11} + x_2 x_4 x_8 x_{10} x_{11} \\
&+ x_2 x_6 x_8 x_{10} x_{11} + x_2 x_4 x_9 x_{10} x_{11} + x_2 x_8 x_9 x_{10} x_{11} \\
&+ x_6 x_8 x_9 x_{10} x_{11} + x_2 x_8 x_{10}^2 x_{11} + x_6 x_8 x_{10}^2 x_{11} \\
&+ x_2 x_4 x_{10} x_{11}^2 + x_2 x_8 x_{10} x_{11}^2 \end{align*}

\begin{theorem} \label{heyguesswhatsplits} A simple 3-connected graph splits if and only if it has  no minor in $\mathcal{F}_0$.
\end{theorem}

\begin{proof} It follows from the above discussion that every graph with a minor in $\mathcal{F}_0$ is non-splitting.  Now let $G$ be a simple 3-connected graph with no minor in $\mathcal{F}_0$.  By Theorem \ref{decomp} we may choose an edge ordering $e_1, \ldots, e_m$ of $G$ of width 3.  Now let 
$S \subseteq E(G)$ be a 5-configuration and choose $1 \le \ell \le m$ so that $e_{\ell}$ is the third edge of this 5-configuration in the ordering.  Now 
both $\partial( \{e_1, \ldots, e_{\ell - 1} \})$ and $\partial( \{e_1, \ldots, e_{\ell} \} )$ have size three.  If these sets are equal, then $G / e_{\ell}$ has a 2-separation with two edges in $S$ on each side.  Otherwise each of these sets has exactly one endpoint of $e_{\ell}$ and $G \setminus e_{\ell}$ has a 2-separation with two edges in $S$ on each side.  In either case, it follows from Corollary \ref{split-tree-cor} that $S$ splits.
\end{proof}

\section{Matroids}\label{sec matroids}

In this section we introduce matroid theory to the study of splitting.  This section is divided into three subsections.  The first is a primer on matroids which we intend as an introduction for those not familiar with their theory.  This subsection can safely be skipped by those readers familiar with matroids.  The second subsection introduces a notion of width for matroids which is quite closely related to the width we have already established for graphs, and compares these two notions.  The final subsection calls upon the Matroid Intersection Theorem to reformulate the notion of splitting in terms of basic connectivity properties.

\subsection{Introduction}

Our primary goal in this subsection is to acquaint the reader with matroids, and to highlight a couple of powerful theorems from this realm with particularly broad scope for application.  

A matroid $M$ consists of a finite ground set $E$ together with a collection $\mathcal{I}$ of subsets of $E$ called \emph{independent sets}, 
which satisfy the following axioms:
\begin{enumerate}
\item $\emptyset \in \mathcal{I}$.
\item If $A \in \mathcal{I}$ and $A' \subseteq A$ then $A' \in \mathcal{I}$.  
\item If $A,B \in \mathcal{I}$ and $|A| < |B|$ then there exists $b \in B \setminus A$ so that $A \cup \{b\} \in \mathcal{I}$.  
\end{enumerate}

Two of the principle examples of matroids are:

\begin{itemize}
\item Let $E$ be a finite set of vectors from a vector space and define $S \subseteq E$ to be independent if these vectors are linearly independent.  
\item Let $E$ be the edge set of a graph $G$ and define $S \subseteq E$ to be independent if these edges do not contain (the edge set of) a cycle.  
This matroid, denoted $M(G)$, is called the \emph{cycle matroid} of $G$.
\end{itemize}

The fact that the first example above satisfies the axioms for a matroid is immediate.  Indeed we may view matroids as a natural abstraction of the concept 
of linear independence in vector spaces.  Our second example is actually a special case of the first.  To see this, let $G = (V,E)$ and define $M$ to be the $V \times E$ incidence matrix of $E$ viewed as a matrix over $\mathbb{F}_2$.  Now associating each edge  with the corresponding column of $M$ we find that a set of edges $S \subseteq E$ has no cycle if and only if the corresponding column vectors are linearly independent.  

A \emph{basis} of a matroid is a maximal independent set.  Generalizing a familiar property from linear algebra, we observe that the third axiom implies that any two bases must have the same size.  Note that for a connected graph, the bases of its cycle matroid are edge sets of spanning trees.  Pleasingly, every matroid $M$ has a dual $M^*$ with the property that $B$ is a basis of $M$ if and only if $E \setminus B$ is a basis of $M^*$.  This duality generalizes that found in planar graphs in the sense that the cycle matroids associated with two dual planar graphs will be dual matroids. 

Another concept from linear algebra which generalizes naturally to matroids is that of rank.  For an arbitrary set $S \subseteq E$, the \emph{rank} of $S$, denoted $r(S)$, is the size of the largest independent set in $S$.  We define the \emph{rank} of the matroid to be $r(M) = r(E)$.  The following classical theorem shows that both covering and packing by bases are well-characterized in terms of the rank function.

\begin{theorem}
Let $M$ be a matroid with ground set $E$ and let $k \ge 0$.  The following hold:
\begin{itemize}
\item Either there exist $k$ bases with union $E$ or there is a subset $S \subseteq E$ with $|S| > k \cdot r(S)$.  
\item Either there exist $k$ disjoint bases in $E$ or there is a subset $S \subseteq E$ so that $|E \setminus S| < k \cdot (r(M) - r(S))$.
\end{itemize}
\end{theorem}

In each of the above cases, the ``or'' is actually an exclusive or, since the two conditions are mutually exclusive.  To see this, we need only note 
that every basis contains at most $r(S)$ elements from a set $S$ (so must contain at least $r(M) - r(S)$ elements from $E \setminus S$). 

Next we state another famous and extremely useful theorem concerning finding a common independent set in two matroids.  This is the
result we will require for our characterization of splitting 5-configurations.

\begin{theorem}[Matroid Intersection]
Let $M_1$ and $M_2$ be matroids on $E$ with rank functions $r_1$ and $r_2$ and let $k \ge 0$.  Exactly one of the following holds:
\begin{enumerate}
\item There exists a set $S \subseteq E$ with $|S| = k$ which is independent in both $M_1$ and $M_2$.
\item There exists a pair of disjoint sets $A,B \subseteq E$ with $A \cup B = E$ so that  $r_1(A) + r_2(B) < k$.
\end{enumerate}
\end{theorem}

\subsection{Matroid separations and caterpillar width}

In this section we introduce the notion of a separation of a matroid, and introduce a notion of width for matroids.  

\begin{definition} Let $M$ be a matroid on $E$ with rank function $r$. A \emph{separation} of $M$ is a pair $(A,B)$, $A,B \subseteq E$ such that $A \cap B = \emptyset$ and $A \cup B = E$. The \emph{order} of the separation is $r(A) + r(B) - r(M) + 1$. 
\end{definition}

Now let us consider a connected graph $G=(V,E)$ and its cycle matroid $M(G)$.  Let $A,B \subseteq E$ satisfy $A \cap B = \emptyset$ and $A \cup B = E$.  Then $(A,B)$ is a separation of the graph $G$ with order $|V(G_A) \cap V(G_B)|$.  Further, $(A,B)$ is a separation of the matroid $M(G)$ with order $|V(G_A) \cap V(G_B)| + 2 - \comp(G_A) - \comp(G_B)$ where $\comp$ gives the number of connected components (this follows from the rank formula $r(S) = |V| - \comp(V,S)$ for every $S \subseteq E$).  So, whenever $A,B \neq \emptyset$ the order of $(A,B)$ in the graph $G$ is greater than or equal to the or its order in the matroid $M(G)$.  

Next we introduce the important notion of branch width.  If $T$ is a tree, we say that $T$ is \emph{cubic} if every vertex of $G$ has degree $1$, or $3$.

\begin{definition}
A \emph{branch decomposition} of a matroid $M$ with ground set $E$ consists of a cubic tree $T$ and an injection $\phi : E \rightarrow V(T)$ with 
range a subset of leaf vertices.  If $T_1,T_2$ are the two components of $T \setminus e$, then the \emph{width} of $e$ is the order of the
separation $( \phi^{-1} (V(T_1)), \phi^{-1}(V(T_2)) )$.  The \emph{width} of $T$ is the maximum width over all of its edges.  
The \emph{branch width} of $M$ is the minimum width over all branch decompositions of $M$.
\end{definition}

For a graph $G$ with cycle matroid $M(G)$, the standard graph theoretical notion of tree width of $G$ is closely related to the branch width of $M(G)$ \cite{RSX}.
The problem of this paper naturally gives rise to decompositions which are linear, rather than tree-like.  The most linear cubic trees are a type of \emph{caterpillar} graphs, that is, trees whose nonleaves form a path.  The \emph{caterpillar width} of a matroid is the minimum width of branch decomposition whose underlying tree is a cubic caterpillar.

\begin{figure}[h]
  \centering
    \includegraphics[scale=1.0]{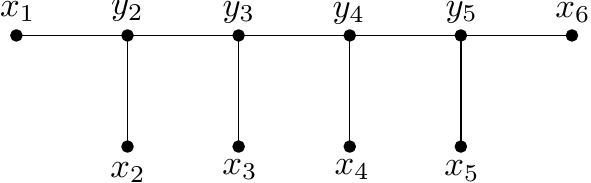}
  \caption{A caterpillar}
\label{caterpillar1}
\end{figure}

\begin{theorem}
If $G = (V,E)$ is a connected graph with $|E| \ge 2$, then the caterpillar width of $M(G)$ is at most the width of $G$.    
\end{theorem}

\begin{proof}
Suppose $k$ is the width of $G$ and let $e_1, \ldots, e_m$ be an ordering of $E$ with width $k$.  If $|E| \ge 3$ then we construct a cubic caterpillar $F$ with vertex set $\{ x_1, \ldots, x_m \} \cup \{ y_2, \ldots, y_{m-1} \}$ and edges $y_i y_{i+1}$ for $2 \le i \le m-2$ and $x_i y_i$  for  $2 \le i \le m-1$ and also $x_1 y_2$ and $x_m y_{m-1}$.  If $|E| = 2$ then we let $F$ be the one edge graph with vertex set $x_1,x_2$.  Now define a branch decomposition of $M(G)$ by mapping $E$ to the leaves of $F$ by the rule that $e_i$ maps to $x_i$.  

For every edge $y_i y_{i+1}$ of $F$, the width of this edge is the order of the separation $(\{e_1, \ldots, e_i\}, \{e_{i+1} , \ldots, e_m \} )$ in $M(G)$ which is at most its order in $G$, which is at most $k$.  Every other edge of $F$ corresponds to a separation of the form $(\{e_i\}, E \setminus \{e_i\})$ for some $1 \le i \le m$ and therefore has order at most $2$.  Since $k \ge 1$, we have the desired outcome unless $k=1$ and there exists an edge $e_i$ for which $(\{e_i\}, E \setminus \{e_i\})$ has order 2 in the matroid.  However, if $k = 1$ then $G$ cannot contain a cycle, so $G$ itself is a tree, but in this case the order of $(\{e_i\}, E \setminus \{e_i\})$ in $M(G)$ will be 1 for every $e_i$.  
\end{proof}

On the other hand, there do exist graphs $G$ for which the width of $G$ is greater than the caterpillar width of $M(G)$.  One such graph is depicted in Figure \ref{k13s}.  This graph has caterpillar width 1 (in fact every such branch decomposition achieves this) but has width 2.  To see this latter claim, note that any edge ordering which does not begin with the form $e_i', e_i$ has order at least two.  Now whatever edge appears next gives rise to a separation of order at least two.  

\begin{figure}[h]
  \centering
    \includegraphics[scale=1.0]{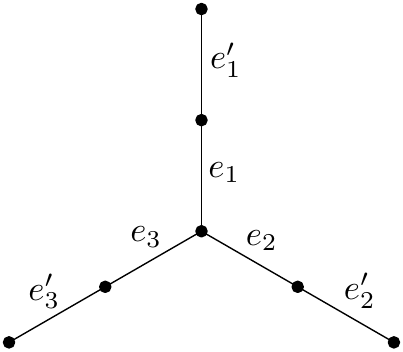}
  \caption{A small tree}
\label{k13s}
\end{figure}

Note that whenever a graph $G$ has an edge ordering $e_1, \ldots, e_m$ of minimum width which has the added property that both 
$G_{ \{e_1, \ldots, e_i\} }$ and $G_{ \{e_{i+1}, \ldots, e_m\} }$ are always connected, the width of $G$ and the caterpillar width of $M(G)$ will be equal.  
In fact, this phenomena is closely related to the difference between two other graph parameters.  Just as the caterpillar width is a linear version of branch width, the path width of a graph is a linear version of tree width.  In the study of path width, there is a somewhat stronger notion, that of connected path width (where the graphs on either side of any separation must always be connected), and the discrepancy between the width of $G$
and caterpillar width of $M(G)$ is closely related to the discrepancy between the path width and connected path width of $G$.

\subsection{A matroidal look at splitting} \label{matroid splitting}

In this section we use the matroid intersection theorem to show that certain natural necessary conditions for non-splitting  are in fact sufficient.  
In fact, the matroid intersection theorem immediately gives us a necessary and sufficient condition in terms of matroidal order for an arbitrary 
Dodgson polynomial to vanish.  Here we will translate this into a more natural graph theoretic condition which will be easy to work with.

\begin{theorem} \label{edmondssplitting} Let $G = (V,E)$ be a connected graph and let $S \subseteq E$ with $|S| = 4$. There is a partition $S = S_1 \cup S_2$, $|S_1| = |S_2| =2$ such that $\Psi^{S_1,S_2} = 0$ if and only if there exist edge disjoint subgraphs $G_1$ and $G_2$ with $G_1 \cup G_2 = G$ such that one of the following holds; 
\begin{enumerate} 
\item $|V(G_1) \cap V(G_2)| \leq 2$ and $|E(G_1) \cap S| = 2$.
\item $|V(G_1) \cap V(G_2)| = 1$ and $|E(G_1) \cap S| = 1$. \end{enumerate} 
\end{theorem}

\begin{proof} If either numbered condition holds, then there cannot not exist a partition $\{S_1,S_2\}$ of $S$ and a set $T \subseteq E \setminus S$
so that both $T \cup S_1$ and $T \cup S_2$ are edge sets of spanning trees.  It then follows from Proposition \ref{calcdodgson} that $\Psi^{S_1, S_2}$ is zero.  

Now, choose a partition of $S$ that meets the stated conditions. Then, there does not exist a $T \subseteq E \cut S$ such that $T \cup S_1$ and $T \cup S_2$ are spanning trees. By assumption, $G$ is connected, so setting $|V(G)| = n$ and $M = M(G)$ we have $r(M) = n-1$. Define matroids $M_1 = M \contract S_1 \cut S_2$ and $M_2 = M \contract S_2 \cut S_1$. If either $S_1$ or $S_2$ contains an edge cut or cycle, the proof is complete. If we assume not, then both of these are cycle matroids on connected graphs with $n-2$ vertices, and thus have rank $n-3$. Furthermore, it follows immediately from our choice of which edges to contract that a set $T \subseteq E \cut S$ satisfies $T \cup S_i$ a spanning tree of $G$ if and only if $T$ is a basis of $M_i$ for $i = 1,2$. By Edmonds' Matroid Intersection Theorem, there exist sets A and B with $A \cap B = \emptyset$ and $A \cup B = E \cut S$ so that $r_1(A) + r_2(B) < n-3$. 

Setting $A^* = A \cup S_1$ and $B^* = B \cup S_2$, we have that $(A^*,B^*)$ is a separation in the original matroid $M$ satisfying \begin{align*} o(A^*,B^*) &= r(A^*) + r(B^*) - r(M) +1 \\ &= (r_1(A)+2) + (r_2(B)+2) - (n-1)+1 \\ &< 3. \end{align*} So, there exists a separation $(A^*,B^*)$ of order at most two (in the matroid) such that $A^*$ and $B^*$ each contain exactly two edges in $S$. Among all such separations, choose one $(A^*,B^*)$ such that $\comp(G_{A^*}) + \comp(G_{B^*})$ is minimum. Let $X = V(G_{A^*}) \cap V(G_{B^*})$, $F_1,...,F_l$ be the components of $G_{A^*}$, and $H_1,...,H_m$ be the components of $G_{B^*}$. By the formula for the order of the separation, we have $|X| \leq l+m$. As $G$ is connected and $F_1,...,F_l$ are distinct components of $G_{A^*}$, the sets $X\cap V(F_1)$ through $X \cap V(F_l)$ are disjoint and nonempty. Suppose there is a component with $|X \cap V(F_i)| =1$. If $E(F_i) \cap S \neq \emptyset$, then the graph induced by $E(G) - E(F_i)$ and $F_i$ must satisfy either (1) or (2). Otherwise, replacing $A^*$ by $A^* \cut E(F_i)$ and $B^*$ by $B^* \cup E(F_i)$ contradicts our choice of $(A^*,B^*)$. So, we may assume that every $F_i$ contains at least two vertices of $X$, and similarly every $H_j$ contains at least two vertices in $X$. By disjointedness and $|X| \leq l +m$, we now have $l = m$, each $F_i$ and $H_j$ contains precisely two points in $X$, and further the vertices in $X$ may be cyclically ordered as $v_1,v_2,...,v_{2m}$ so that every $F_i$ contains $v_{2i-1}$ and $v_{2i}$ and every $H_i$ contains $v_{2i}$ and $v_{2i+1}$ modulo $2m$. As $A^*$ and $B^*$ both contained precisely two edges of $S$, each $F_i$ and $H_j$ may contain at most two edges in $S$, and hence there must be subgraphs $G_1$ and $G_2$ satisfying the first property above. This completes the proof. \end{proof}

The following corollary relates Theorem \ref{edmondssplitting} to our interests in Dodgson polynomials and 5-configurations.

\begin{corollary} \label{whatsplitmeans} Let $G$ be a graph and $S$ a 5-configuration. Then $S$ splits if and only if there exists an $e \in S$ so that, setting $G'$ equal to one of $G \cut e$ or $G \contract e$, there exist edge-disjoint subgraphs $G_1$ and $G_2$ of $G'$ with $G' = G_1 \cup G_2$ satisfying one of the following; \begin{enumerate} \item $|V(G_1) \cap V(G_2)| \leq 2$ and $|E(G_1) \cap S| = 2$. \item $|V(G_1) \cap V(G_2)| = 1$ and $|E(G_1) \cap S| = 1$. \end{enumerate} \end{corollary}

With this corollary in hand, we can give a simple proof that the graphs in $\mathcal{F}_0$ are non-split.  All of these graphs have the property 
that whenever $(A,B)$ is a 3-separation, $\min\{ |A|, |B| \} \le 3$.  It follows from this and Corollary \ref{whatsplitmeans} that whenever $S$ is any set 
of edges in one of these graphs which does not contain a triangle or all three edges incident with a vertex of degree 3, then $S$ does not 
split.  So, in particular, the 5-configurations highlighted in Figure \ref{nonsplitfives} do not split.

\section{Enhanced Graphs}\label{wands}

As we shall shortly prove, every minor-minimal non-split graph is rather close to being 3-connected in the sense that it may only have
a small number of specially structured 2-separations.  It is natural then to replace the small side of such a separation by a single edge encoded with some added information so as to operate in the setting of 3-connected graphs.  In this section we adopt this approach, and reduce our 
problem to that of finding excluded minors in this new setting.

An \emph{enhanced graph} consists of a graph $\widetilde{G}$ together with two distinguished subsets 
$\widetilde{C}, \widetilde{D} \subseteq E( \widetilde{G} )$.  The edges in $\widetilde{C}$ are called \emph{contract-proof} and the edges in 
$\widetilde{D}$ are \emph{delete-proof}.  As before, a set $\widetilde{S}$ of five edges is called a 5-configuration,
However, we shall define splitting differently for enhanced graphs.  We say that a 5-configuration $\widetilde{S}$ \emph{splits} if there is an
enhanced graph which is either equal to $\widetilde{G}$ or obtained from it by deleting an edge in $\widetilde{S} \setminus \widetilde{D}$ or 
contracting an edge in $\widetilde{S} \setminus \widetilde{C}$ which has a separation $(A,B)$ satisfying one of the following:
\begin{itemize}
\item $(A,B)$ has order at most 1 and $A \cap \widetilde{S} \neq \emptyset \neq B \cap \widetilde{S}$.
\item $(A,B)$ has order 2 and $\min\{ |A \cap \widetilde{S}|, |B \cap \widetilde{S}| \} = 2$.
\end{itemize}
Note that in the special case when $\widetilde{C} = \widetilde{D} = \emptyset$ this definition aligns with the notion of splitting in ordinary graphs 
thanks to Corollary \ref{whatsplitmeans}.  We say that a 2-separation $(A,B)$ satisfying the second property above is a \emph{bad} $2$-separation 
(this will generally be the obstruction we encounter).  We will refer to both contract-proof and delete-proof edges as forms of \emph{protection}.

 Let $G$ be an ordinary connected graph and let $S$ be a non-split 5-configuration in $G$.  It follows from the definitions that 
every 1-separation $(A,B)$ of $G$ must satisfy $A \cap S = \emptyset$ or $B \cap S = \emptyset$.  So, there exists a block $G'$ of 
$G$ with $S \subseteq E(G')$.  Now consider $G'$ and note that it is 2-connected and non-split.  It follows from the definitions
that every 2-separation $(A,B)$ of $G'$ must satisfy $\min \{ |A \cap S|, |B \cap S| \} \le 1$.  Define a subset $\emptyset \neq L \subset E(G')$ to be a 
\emph{lobe} if $(L, E(G') \setminus L)$ is a 2-separation with $|L \cap S| \le 1$.  It follows from the 2-connectivity of $G'$ that whenever $L,L'$ are 
lobes with $L \cap L' \neq \emptyset$ the set $L \cup L'$ is also a lobe.  Furthermore, every edge $e$ must be contained in a lobe, since 
$( \{e\}, E(G') \setminus \{e\} )$ is a 2-separation.  Therefore, the maximal lobes give us a partition of $E(G')$.  Next 
we construct an enhanced graph $\widetilde{G}$ with $\widetilde{C}$, $\widetilde{D}$ and a 5-configuration $\widetilde{S}$ as follows.  We 
begin with $\widetilde{G} = G'$ and with $\widetilde{C} = \widetilde{D} = \widetilde{S} = \emptyset$.  Now, for every maximal lobe $L$ in the graph $G'$ with $\partial(L) = \{x,y\}$, we replace $L$ in $\widetilde{G}$ by the single edge $xy$.  
If $L \cap S \neq \emptyset$ then we add the edge $xy$ to $\widetilde{S}$.  If in addition $L \setminus S$ contains a path from $x$ to $y$ then we add $xy$ to $\widetilde{D}$ and if the unique edge in $S \cap L$ does not have ends $x,y$ then we add $xy$ to $\widetilde{C}$.  We say that 
the enhanced graph $\widetilde{G}$ together with $\widetilde{S}$ are \emph{associated with} $G$ and $S$.  The following proposition captures the 
key properties of this operation.

\begin{prop}
Let $S$ be a non-splitting 5-configuration in the connected graph $G$ and let $\widetilde{G}$ together with $\widetilde{S}$ be associated with $G$ and $S$.  Then we have:
\begin{enumerate}
\item $\widetilde{G}$ is 3-connected.
\item $\widetilde{S}$ does not split in $\widetilde{G}$.
\end{enumerate}
\end{prop}

\begin{proof} It follows immediately from the construction that $\widetilde{G}$ is 3-connected.  Therefore, the enhanced graph 
$\widetilde{G}$ has no bad 2-separation.  If $e \in \widetilde{S}$ is not delete-proof and $\widetilde{G} \setminus e$ has a bad 2-separation, then
it follows immediately that the graph obtained from $G$ by deleting the edge in $S$ associated with $e$ also has a bad 2-separation.  
A similar argument for contraction reveals that $\widetilde{S}$ does not split in $\widetilde{G}$, as desired. \end{proof}

\begin{center}
\begin{figure}[h]
\centerline{\includegraphics[width=10cm]{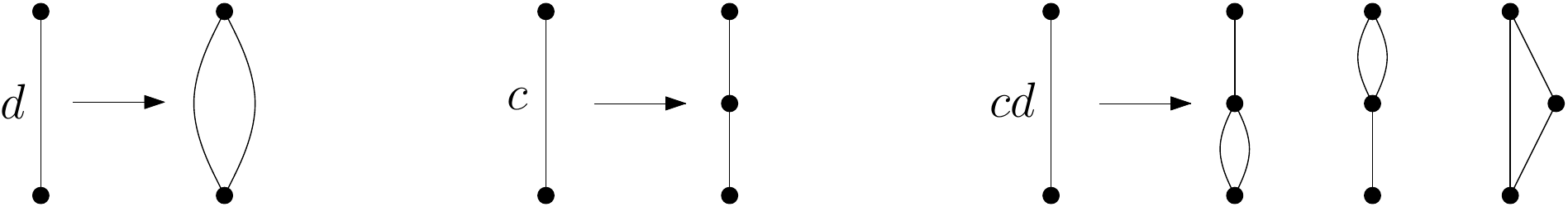}}
\caption{}
\label{ordinary-enhanced}
\end {figure}
\end{center}

We can also reverse this process.  Let $\widetilde{S}$ be a non-split 5-configuration in the enhanced graph $\widetilde{G}$.  Now, form an ordinary graph $G$ and 5-configuration $S$ as follows.  We begin with $G = \widetilde{G}$ and $S = \emptyset$.  Now for every edge $e = xy \in \widetilde{S}$ we proceed as follows: if $e \not\in \widetilde{C} \cup \widetilde{D}$ then we add $e$ to $S$.  If $e \in \widetilde{D} \setminus \widetilde{C}$ then we replace $e$ by a pair of edges in parallel between $x$ and $y$ and add one of these to $S$.  If $e \in \widetilde{C} \setminus \widetilde{D}$, then we subdivide $e$ and add one of these two newly formed edges to $S$.  Finally, if 
$e = xy \in \widetilde{C} \cap \widetilde{D}$ then we replace the edge $e$ by one of the three graphs shown in the last transformation of Figure \ref{ordinary-enhanced}, call it $F$, and add an edge $f \in E(F)$ to $S$ so that $f \neq xy$ and so that $F \setminus f$ contains a path from $x$ to $y$.  We say that $G,S$ are \emph{associated with} $\widetilde{G},\widetilde{S}$.  

\begin{prop}
Let $\widetilde{S}$ be a non-split 5-configuration in the enhanced graph $\widetilde{G}$.  If $G,S$ is associated with $\widetilde{G}, \widetilde{S}$, then $S$ is a non-split 5-configuration in $G$.
\end{prop}

\begin{proof} This follows from the structure of $G$ and the definition of splitting.
\end{proof}

It follows immediately from the previous two propositions that whenever $(A,B)$ is a non-trivial 2-separation in a minor-minimal non-split (ordinary) 
graph $G$, one of the sets must either consist of two edges in parallel or in series with exactly one in $S$, or must consist of three edges arranged
as in one of the three graphs on the right in Figure \ref{ordinary-enhanced}.  

To take full advantage of enhanced graphs, we will need to allow a slight extension of the usual graph minor operations in this setting.  
A \emph{minor} of an enhanced graph $\widetilde{G}$ is an enhanced graph obtained from $\widetilde{G}$ by any sequence of the following operations:
\begin{enumerate}
\item Vertex deletion.
\item Edge deletion and contraction.
\item removal of an edge from either $\widetilde{C}$ or $\widetilde{D}$ while leaving it in the graph.
\item If $e,e'$ are parallel edges,  delete $e'$ and add $e$ to $\widetilde{D}$.
\item If $e,e'$ are the edges incident with a vertex of degree two,  contract $e'$ and add $e$ to $\widetilde{C}$.
\end{enumerate}
These operations respect our framework, as indicated by the following easy proposition.

\begin{prop}
Let $G$ be a graph with a non-split 5-configuration $S$ and let $\widetilde{G}$ together with $\widetilde{S}$ be the associated enhanced graph and 5-configuration.  Then $G$ is minor-minimal subject to $S$ not splitting if and only if $\widetilde{G}$ is minor-minimal subject to $\widetilde{S}$ not splitting.
\end{prop}

This permits us to work in the setting of enhanced graphs, and this offers us a couple of key advantages.  First, the enhanced graphs
of interest to us will be 3-connected, and this has technical advantages.  Second, one non-splitting enhanced graph may in general give rise to many 
non-splitting graphs, so working in the context of enhanced graphs gives us a more compact description of the minimal forbidden graphs.  

We are now ready to introduce the family of enhanced graphs which we will prove are the excluded minors for non-splitting (i.e. the full family of minor-minimal
 non-split enhanced graphs).  
Define $\mathcal{F}$ to be the set of enhanced graphs consisting of the graphs in $\mathcal{F}_0$ (with no contract or delete-proof edges), 
together with the enhanced graphs appearing in Appendix \ref{app all_excl} (here an edge marked ``c'' is contract-proof and an edge marked ``d'' 
is delete-proof).  The main result established over the course of the following two sections is as follows.

\begin{theorem}
\label{mainskeleton}
The set $\mathcal{F}$ consists of all minor minimal non-split enhanced graphs.
\end{theorem}

\begin{corollary}
The minor minimal non-split graphs are precisely all graphs associated with an enhanced graph in $\mathcal{F}$.
\end{corollary}

Since we shall now focus on enhanced graphs, let us close this section with a few simple properties of non-split 5-configurations in enhanced graphs.

\begin{lemma}
\label{mm3sep}
Let $\widetilde{G}$ be an enhanced graph and let $\widetilde{S}$ be a non-split 5-configuration.
\begin{enumerate}
\item If $\widetilde{S}$ contains a triangle, then all of these edges must be contract-proof.
\item If $\widetilde{S}$ contains all three edges incident with some vertex of degree three, then all of these edges must be delete-proof.
\item If $\widetilde{G}$ is minor-minimal subject to  $\widetilde{S}$ being non-split, then $\widetilde{C} \cup \widetilde{D} \subseteq \widetilde{S}$.  
\item If $\widetilde{G}$ is minor-minimal subject to $\widetilde{S}$ being non-split, then every separation $(A,B)$ of $\widetilde{G}$ of order at most 3 with $A \cap \widetilde{S} = \emptyset$ must satisfy $|A| \le 3$.
\end{enumerate}
\end{lemma}

Enhanced graphs also have a natural concept of planar duality (as suggested by the first two properties in the preceding lemma) which we will exploit to simplify our proof.  If $\widetilde{G}$ together with $\widetilde{C},\widetilde{D}$ is an enhanced graph embedded in the plane, then we define its \emph{dual} to be the enhanced graph consisting of the planar dual of $\widetilde{G}$ (on the same edge set), which we denote by $(\widetilde{G})^*$, together with the contract-proof edges $\widetilde{C}^* = \widetilde{D}$ and the delete-proof edges $\widetilde{D}^* = \widetilde{C}$.  So, the dual of a contract-proof edge is delete-proof, and vice-versa.  The next proposition shows that the notion of splitting is invariant under duality.

\begin{prop}\label{enhanced dual}
Let $\widetilde{G}$ be an enhanced planar graph and $(\widetilde{G})^*$ its dual.  Then a 5-configuration $\widetilde{S}$ splits in $\widetilde{G}$ if and only if it splits in $(\widetilde{G})^*$.  
\end{prop}

\begin{proof} This follows immediately from our definitions and the fact that the order of a separation $(A,B)$ in $\widetilde{G}$ is the same as the order of the corresponding separation in $(\widetilde{G})^*$.  
\end{proof}

\section{Excluded minors}

Throughout this section and the next we work exclusively with enhanced graphs.  The purpose of these two sections is to prove 
that $\mathcal{F}$ is the set of all minor-minimal non-split enhanced graphs (Theorem \ref{mainskeleton}).  Our proof of this divides
naturally into two parts.  First we show that the enhanced graphs in $\mathcal{F}$ do not split, and no one has another as a minor.  
Then we will show that every non-split enhanced graph has a minor in $\mathcal{F}$.  These are the respective purposes of this 
and the next section.  Let us note here that this list $\mathcal{F}$ of excluded minors was arrived at both by analysis of minimal counterexample
and independently by way of a computer search which was exhaustive up to all enhanced graphs with 11 edges.

\subsection{Wheels}\label{secwheel}

In this subsection we investigate those enhanced graphs which are isomorphic to a wheel.  For $k \ge 3$ we define the $k$-\emph{wheel}, denoted $W_k$, to be the simple graph obtained from a cycle of length $k$ by adding a single new vertex, called the \emph{centre} vertex, adjacent to all existing vertices.  We call the edges on this original cycle \emph{rim edges} and the newly added edges \emph{spoke edges}.  Note that the dual of a wheel is another wheel isomorphic to the original where the roles of rim edges and spoke edges have been reversed.  We begin with a key observation.

\begin{observation}
\label{easy_break}
Every $e \in E(W_k)$ satisfies:
\begin{enumerate}
\item If $e$ is a rim edge, then $G \setminus e$ has width $2$.
\item If $e$ is a spoke edge, then $G / e$ has width $2$.
\end{enumerate}
\end{observation}

It is an immediate consequence of the above observation that whenever $\widetilde{G}$ is a wheel with a non-split 5-configuration $\widetilde{S}$, 
every rim edge in $\widetilde{S}$ must be delete-proof, and every spoke edge in $\widetilde{S}$ must be contract proof.

\begin{center}
\begin{figure}[ht]
\centerline{\includegraphics[height=3cm]{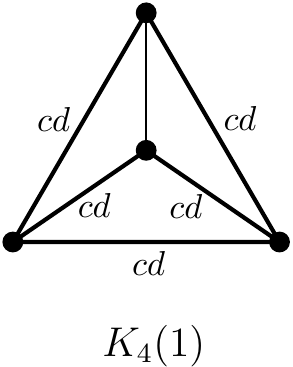}}
\caption{Excluded Minors isomorphic to $K_4$}
\label{excl_k4}
\end {figure}
\end{center}

\begin{lem}
\label{k4}
The unique excluded minor for splitting isomorphic to $K_4$ is $K_4(1)$.
\end{lem}

\begin{proof} Let $\widetilde{G}$ be an enhanced graph isomorphic to $K_4$ with a non-split 5-configuration $\widetilde{S}$.  Since $\widetilde{G}$ 
is a wheel with any vertex playing the role of centre, it follows from the previous observation that every edge in $\widetilde{S}$ must be both contract
and delete-proof.  In a minor-minimal non-split enhanced graph, every protected edge must appear in every non-split 5-configuration (otherwise a protection 
could be removed), and the result follows immediately from this.  \end{proof}

\begin{lem}
\label{nonsplit_wheel}
Assume $\widetilde{G}$ is isomorphic to a wheel $W_k$ with $k \ge 4$ and centre vertex $w$.  Then $\widetilde{S}$ is non-splitting if and only if it satisfies all of the following conditions.
\begin{enumerate}
\item Every rim edge in $\widetilde{S}$ is delete-proof,
\item Every spoke edge in $\widetilde{S}$ is contract-proof.
\item If $v$ is a non-centre vertex and all edges incident with $v$ are in $\widetilde{S}$, then $vw$ is delete-proof.
\item If the edges $vw$, $v'w$, $vv'$ are all in $\widetilde{S}$ then $vv'$ is contract-proof.
\end{enumerate}
\end{lem}

\begin{proof} First suppose that $\widetilde{S}$ is non-splitting.  Then Observation \ref{easy_break} implies the first two conditions, 
and the last two follow from Lemma \ref{mm3sep}.  Next assume that $\widetilde{S}$ satisfies the above conditions, and consider 
an edge $e \in \widetilde{S}$.  First suppose $e = vw$ is a spoke edge, and note that the second condition implies it is contract proof.  
Deleting $e$ gives an enhanced graph with a single non-trivial 2-separation (up to switching $(A,B)$ with $(B,A)$), and the third condition implies 
that this is not a bad 2-separation.  Next assume that $e = vv'$ is a rim edge, and note that the first condition implies it is delete-proof.  Contracting $e$ yields an enhanced graph with a single non-trivial 2-separation, and the fourth condition above implies this separation is not bad.  It follows from this analysis that $\widetilde{S}$ does not split, as desired. \end{proof}

Before our next lemma, let us introduce a useful definition.  For an enhanced graph $\widetilde{G}$ with contract-proof and delete-proof 
edges given by the sets $\widetilde{C}, \widetilde{D} \subseteq E(G)$, we define the \emph{weight} of $\widetilde{G}$ to be 
$|E(\widetilde{G})| + |\widetilde{C}| + |\widetilde{D}|$.  So, the weight of $\widetilde{G}$ is 
precisely the number of edges in any graph associated with $\widetilde{G}$.  
Weight will be a nice parameter for us since it is minor monotone, but certain small graphs such as $K_4$ must have large weight in order
to be non-split enhanced graphs.  This section shall feature numerous figures which list excluded enhanced graph minors isomorphic to the same graph $G$.  In these cases, we will, by convention, number these enhanced graphs $G(1), G(2), \ldots$ in order of non-decreasing weight.  Furthermore, when $G$ is self-dual, we will list the enhanced graphs so that $G(i)$ is dual to one of $G(i-1)$, $G(i)$, or $G(i+1)$.  For instance, in the following figure $W_4(2k-1)$ and $W_4(2k)$ are 
dual for every $1 \le k \le 5$.

\begin{center}
\begin{figure}[ht]
\centerline{\includegraphics[scale=0.65]{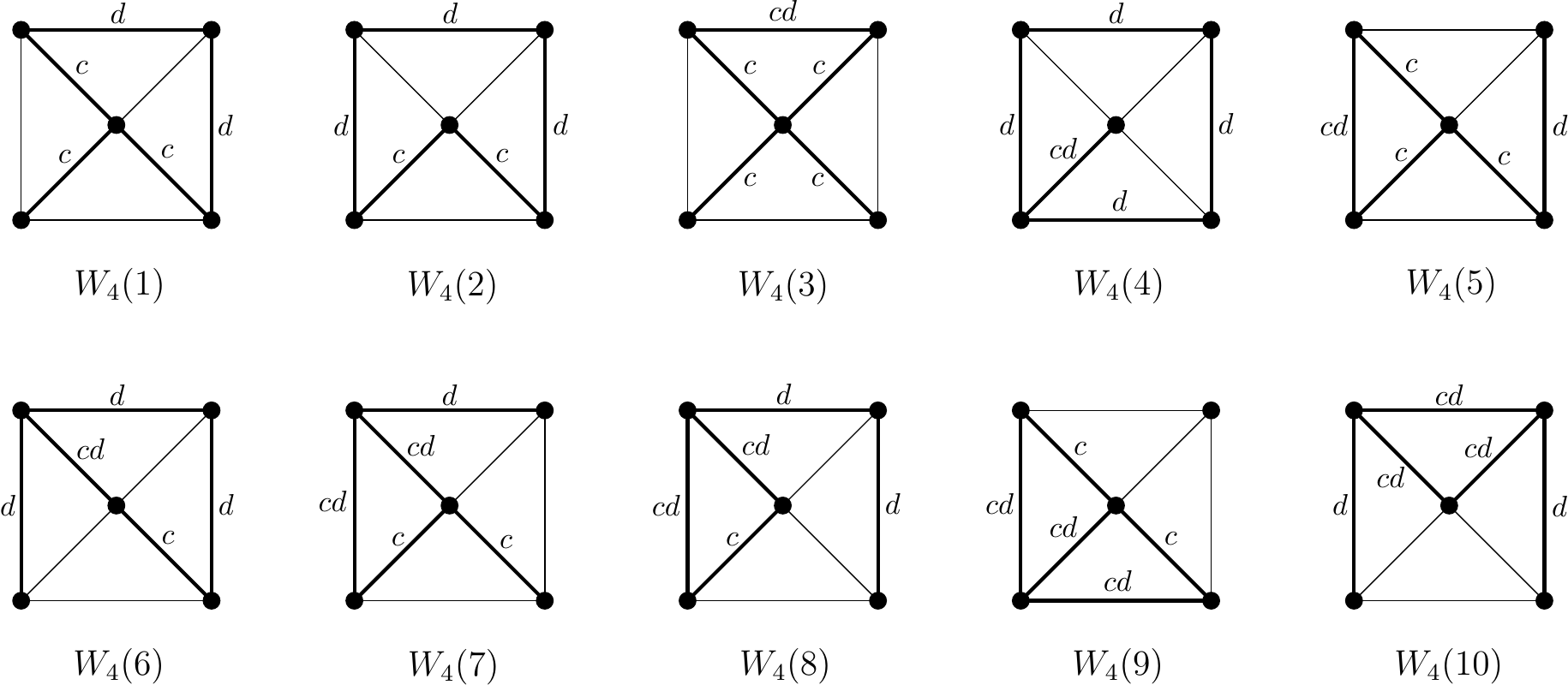}}
\caption{Excluded Minors isomorphic to $W_4$}
\label{excl_w4}
\end {figure}
\end{center}

\begin{lem}
\label{w4}
The excluded minors isomorphic to $W_4$ are $W_4(1), \ldots, W_4(10)$.
\end{lem}

\begin{proof} For each enhanced graph $W_4(i)$ as given in the figure, we associate the 5-configuration indicated by the bold edges
(these are also precisely the edges which are either contract-proof or delete-proof).  It follows from inspection that these are 
precisely all minimally protected enhanced graphs isomorphic to $W_4$ which satisfy the conditions in Lemma \ref{nonsplit_wheel} 
(to see this, note that $W_4(3)$ is the unique such enhanced graph with 4 protected spokes, $W_4(1), W_4(5), W_4(7), W_4(9)$ are those with 
3 protected spokes, $W_4(2), W_4(6), W_4(8), W_4(10)$ are those with 2 protected spokes, and $W_4(4)$ is the unique such 
enhanced graph with 1 protected spoke).  Therefore, every minor-minimal non-splitting enhanced graph with $\widetilde{G} \cong W_4$ appears in our list.  To see that 
they are all minor minimal, note that Lemma \ref{nonsplit_wheel} implies we cannot remove a protection without causing the enhanced graph to split.  Next consider an edge $e \in \widetilde{E} \setminus \widetilde{S}$.  If $e$ is a rim edge, then Observation \ref{easy_break} implies that contracting $e$ 
results in a split enhanced graph.  Deleting $e$ results in an enhanced graph which is isomorphic to a subdivision of $K_4$ but has weight less than 16, which splits by Lemma \ref{k4}. Similarly, if $e$ is a spoke edge, then Observation \ref{easy_break} and Lemma \ref{k4} imply that either deleting or contracting $e$ results in a split enhanced graph. \end{proof}

\begin{center}
\begin{figure}[ht]
\centerline{\includegraphics[scale=0.8]{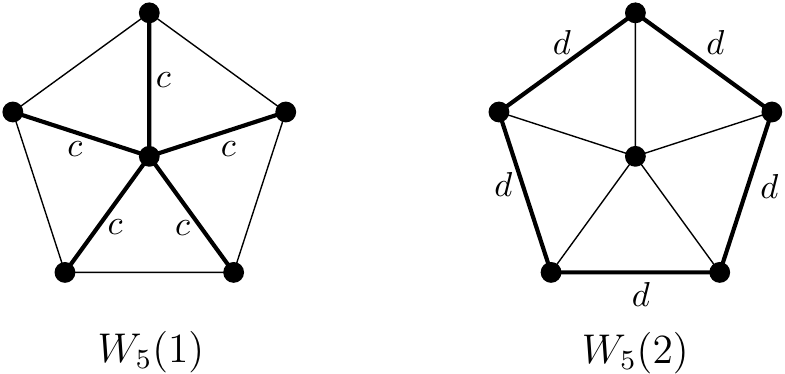}}
\caption{Excluded Minors isomorphic to $W_5$}
\label{excl_w5}
\end {figure}
\end{center}

\begin{lem} 
\label{w5}
The excluded minors isomorphic to $W_k$ for $k \ge 5$ are $W_5(1)$ and $W_5(2)$.
\end{lem}

\begin{proof} It follows from Lemma \ref{nonsplit_wheel} that the two enhanced graphs in Figure \ref{excl_w5} do not split.  Deleting or contracting 
an edge not in $\widetilde{S}$ from either of these enhanced graphs immediately yields a bad 2-separation, so they are indeed minimal.  

Now let $\widetilde{G}$ be an enhanced graph which is isomorphic to a wheel with at least 6 vertices which is minor minimal non-splitting.  
Let $w \in V(\widetilde{G})$ correspond to the centre.  Define two (distinct) edges $e,f \in \widetilde{E}$ to be \emph{close} if they are both adjacent and cofacial.  Note that every edge is close to two rim edges and two spoke edges,  and that this notion is invariant under duality.  

\medskip

\noindent{\it Claim:} If $e \in \widetilde{E} \setminus \widetilde{S}$ is a rim (spoke) edge, then either both spokes (rims) close to $e$ are in $\widetilde{S}$ or there are three edges in $\widetilde{S}$ close to $e$.  

\smallskip

By duality, it suffices to prove the claim when $e$ is a rim edge.  
Let $v_0 v_1, v_1 v_2, v_2 v_3$ be rim edges with $e = v_1 v_2$.  If $w v_1, w v_2 \in \widetilde{S}$ then we are done, so by symmetry we may 
assume $w v_2 \not\in \widetilde{S}$.  Let $\widetilde{G_0}$ denote the enhanced graph obtained from $\widetilde{G}$ by contracting $e = v_1 v_2$ and 
deleting $w v_2$ and then, if $w v_1 \in \widetilde{S}$ we make this edge delete-proof in $\widetilde{G_0}$.  It follows that $\widetilde{G_0}$ is a proper minor of $\widetilde{G}$, so it must split.  It then follows from Lemma \ref{nonsplit_wheel} that both $w v_3$ and $v_2 v_3 \in \widetilde{S}$.  Now let us consider the enhanced graph $\widetilde{G_1}$ obtained from $\widetilde{G}$ by deleting the edge $w v_2$ and contracting the edge $v_1 v_2$ and then making the edge $v_2 v_3$ contract-proof in $\widetilde{G_1}$.  Again $\widetilde{G_1}$ is a proper minor of $\widetilde{G}$, so it must split and then Lemma \ref{nonsplit_wheel} implies that both $v_0 v_1$ and $w v_1 \in \widetilde{S}$.  

\medskip

Suppose that every rim edge in $\widetilde{E} \setminus \widetilde{S}$ has both spokes close to it in $\widetilde{S}$ and let $a$ be the number of 
rim edges in $\widetilde{S}$.  It then follows that either $a= 0$ and $\widetilde{G}$ contains the first enhanced graph in Figure \ref{excl_w5} as a minor, or 
$a=5$ and $\widetilde{G}$ is isomorphic to the second enhanced graph in the figure.  So, we may assume that there exists a rim edge in $\widetilde{E} \setminus \widetilde{S}$ which is close to a spoke in $\widetilde{E} \setminus \widetilde{S}$.  By the claim, we may then assume that $v_i v_{i+1}$ is a 
rim edge for $0 \le i \le 5$ (here $v_0 = v_5$ is permitted) and that $v_1 v_2, w v_2 \not\in \widetilde{S}$ and $v_0 v_1, w v_1, v_2 v_3, w v_3 \in \widetilde{S}$.  Since at most one of $w v_4$, $w v_5$, $v_3 v_4$ is in $\widetilde{S}$, the claim implies $v_4 v_5 \in \widetilde{S}$.  However, now applying the claim to the edge $w v_4$ brings us to a contradiction. \end{proof}

\subsection{$K_5^-$ and Prism}

In this subsection we will determine all excluded minors isomorphic to either $K_5^-$, the (unique up to isomorphism) graph 
obtained from $K_5$ by removing an edge, or $P$ (short for Prism), the planar dual of $K_5^-$.  Thanks to duality, it suffices to find all 
excluded minors with enhanced graph $K_5^-$ and then dualize.  We say that an edge of $K_5^-$ is \emph{longitudinal} if it is incident with a vertex of degree 3 and \emph{equatorial} otherwise.  We begin with a simple lemma.

\begin{lem}
\label{k5m_split}
Let $\widetilde{G}$ be isomorphic to $K_5^-$.  A 5-configuration $\widetilde{S}$ is non-splitting if and only if all of the following conditions are satisfied.
\begin{enumerate}
\item Every equatorial edge in $\widetilde{S}$ is contract-proof.
\item Every longitudinal edge contained  in a triangle within $\widetilde{S}$ is contract-proof.
\item If $uv \in \widetilde{S}$ is longitudinal where ${\mathit deg}(u) = 3$ and the unique triangle containing $u$ but not $uv$ contains 
at least two edges in $\widetilde{S}$, then $uv$ is delete-proof.
\end{enumerate}
\end{lem}

\begin{proof} First suppose that $\widetilde{G}$ satisfies the numbered properties above and let $e \in \widetilde{S}$.  If $e$ is 
equatorial, then it is contract-proof, and $\widetilde{G} \setminus e$ is 3-connected, so it has no bad 2-separation.  If $e$ is longitudinal, 
and not contract-proof, the second property implies $\widetilde{G} / e$ has no bad 2-separation, and if it is not delete-proof, the third property 
implies $\widetilde{G} \setminus e$ has no bad 2-separation.  Thus $\widetilde{G}$ is non-splitting.

If $\widetilde{G}$ violates the first property for $e$, then $\widetilde{G} / e$ has width 2, so $\widetilde{G}$ must split.  If it violates the second or 
third properties, these conditions indicate separations which show $\widetilde{G}$ splits. \end{proof}

\begin{center}
\begin{figure}[ht]
\centerline{\includegraphics[scale=0.8]{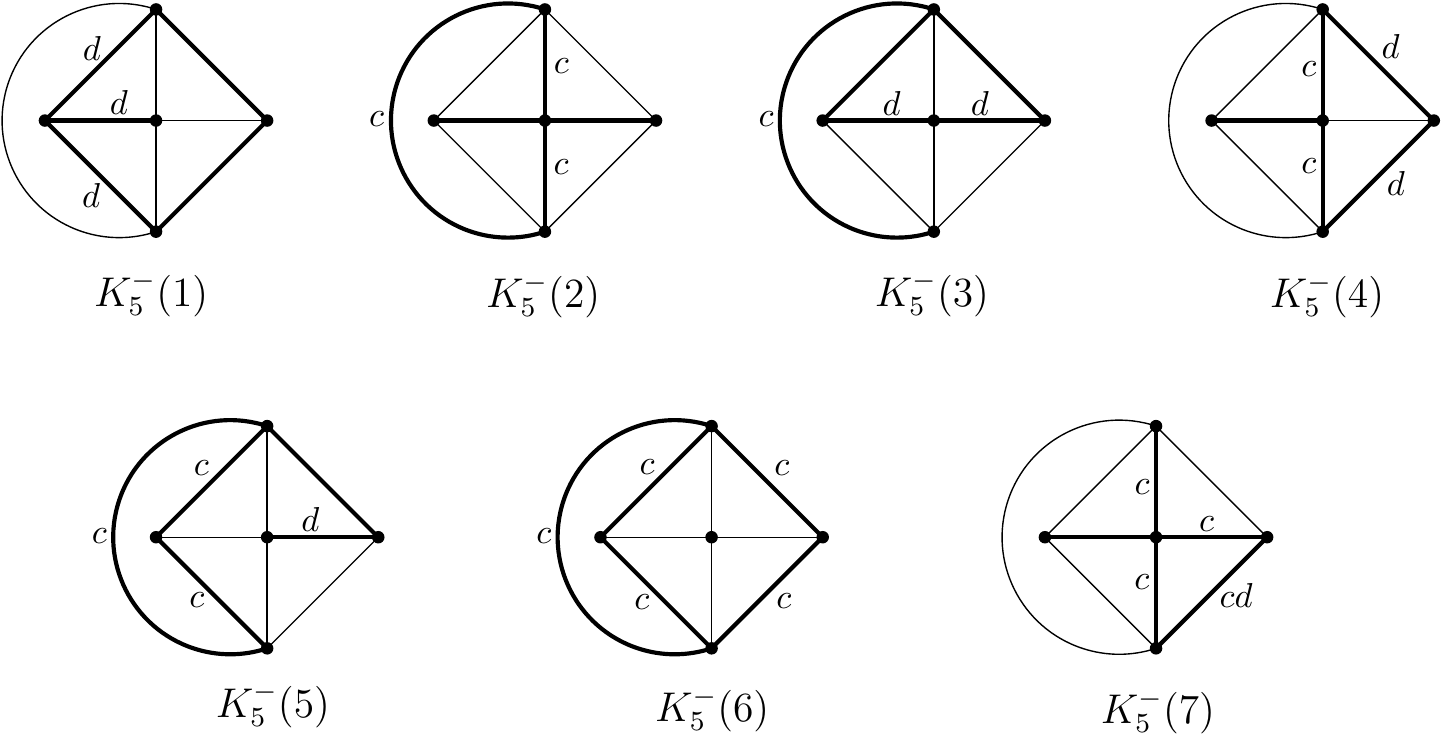}}
\caption{Excluded Minors isomorphic to $K_5^-$}
\label{excl_k5m}
\end {figure}
\end{center}

\begin{lem}
\label{exclk5mlem}
The excluded minors isomorphic to $K_5^-$ are $K_5^-(1), \ldots, K_5^-(7)$.
\end{lem}

\begin{proof} We shall begin by proving that whenever $\widetilde{G}$ is a non-split enhanced graph isomorphic to $K_5^-$, 
then $\widetilde{G}$ contains one of $K_5^-(1), \ldots, K_5^-(7)$ as a minor.  To do this, let us choose a non-slitting 5-configuration 
$\widetilde{S}$ in $\widetilde{G}$.  If $\widetilde{S}$ contains all three edges incident with some vertex of degree 3, then these edges must be
delete-proof by Lemma \ref{k5m_split} so $\widetilde{G}$ contains $K_5^-(1)$.  If $\widetilde{S}$ contains all three equatorial edges, then all 
must be contract-proof so $\widetilde{G}$ contains $K_5^-(2)$.  So, we may now assume $\widetilde{S}$ contains neither of these configurations.  

Next suppose that $\widetilde{S}$ contains two equatorial edges $e,e'$ and two edges $f,f'$ incident with a vertex of degree 3.  If $e,e',f,f'$ forms a four cycle then it follows from Lemma \ref{k5m_split} that $e,e'$ are contract-proof and $f,f'$ are delete-proof so $\widetilde{G}$ contains $K_5^-(4)$.  Otherwise, we may assume $e,f,f'$ form a triangle and $f$ is incident with both $e$ and $e'$.  In this case $f,f'$ must be 
contract-proof, and in addition $f'$ must be delete-proof, so $\widetilde{G}$ contains $K_5^-(7)$.  

In the remaining case $\widetilde{S}$ contains exactly one equatorial edge $e$ and we may assume that it contains exactly two edges $f,f'$ incident 
with one vertex of degree 3, and exactly two edges $h,h'$ incident with the other vertex of degree 3.  If $e,f,f'$ form a triangle, then $f,f'$ must be 
contract-proof, and otherwise one of $f,f'$ must not be adjacent to $e$ and this edge must be delete-proof.  A similar argument for the edges $h,h'$ 
imply that $\widetilde{G}$ must contain one $K_5^-(3)$, $K_5^-(5)$, and $K_5^-(6)$.  

With the above argument in hand, we are ready to complete the proof.  A straightforward application of Lemma \ref{k5m_split} (using the indicated 
5-configurations) shows that $K_5^-(1), \ldots, K_5^-(7)$ do not split.  It follows from the above analysis that every enhanced graph obtained from one of these
by removing a protection does split.  So, to prove the lemma, it suffices to show that every enhanced graph $\widetilde{K}$ obtained from one in these figures by deleting or contracting an unprotected edge will split.  This is immediate whenever $\widetilde{K}$ has width 2 and follows from Lemma \ref{k4} whenever $\widetilde{K}$ can be obtained from $K_4$ by series parallel operations, since $\widetilde{K}$ has weight at most 14.  The only other possibility is that $\widetilde{K}$ is isomorphic to a graph obtained from $W_4$ by series parallel operations.  Now $\widetilde{K}$ has weight at most 12 unless it is obtained from $K_5^-(6)$ or $K_5^-(7)$ by deleting an unprotected equatorial edge.  On the other hand, apart from $W_4(1)$ and $W_4(2)$ which have weight 13, all $W_4(i)$ with $3 \le i \le 10$ have weight at least 14.  So, the only possibility for $\widetilde{K}$ to contain $W_4(i)$ is for one of $K_5^-(6)$ or $K_5^-(7)$ to contain either $W_4(1)$ or $W_4(2)$, and a quick check reveals this is impossible. \end{proof}

\begin{center}
\begin{figure}[ht]
\centerline{\includegraphics[scale=0.7]{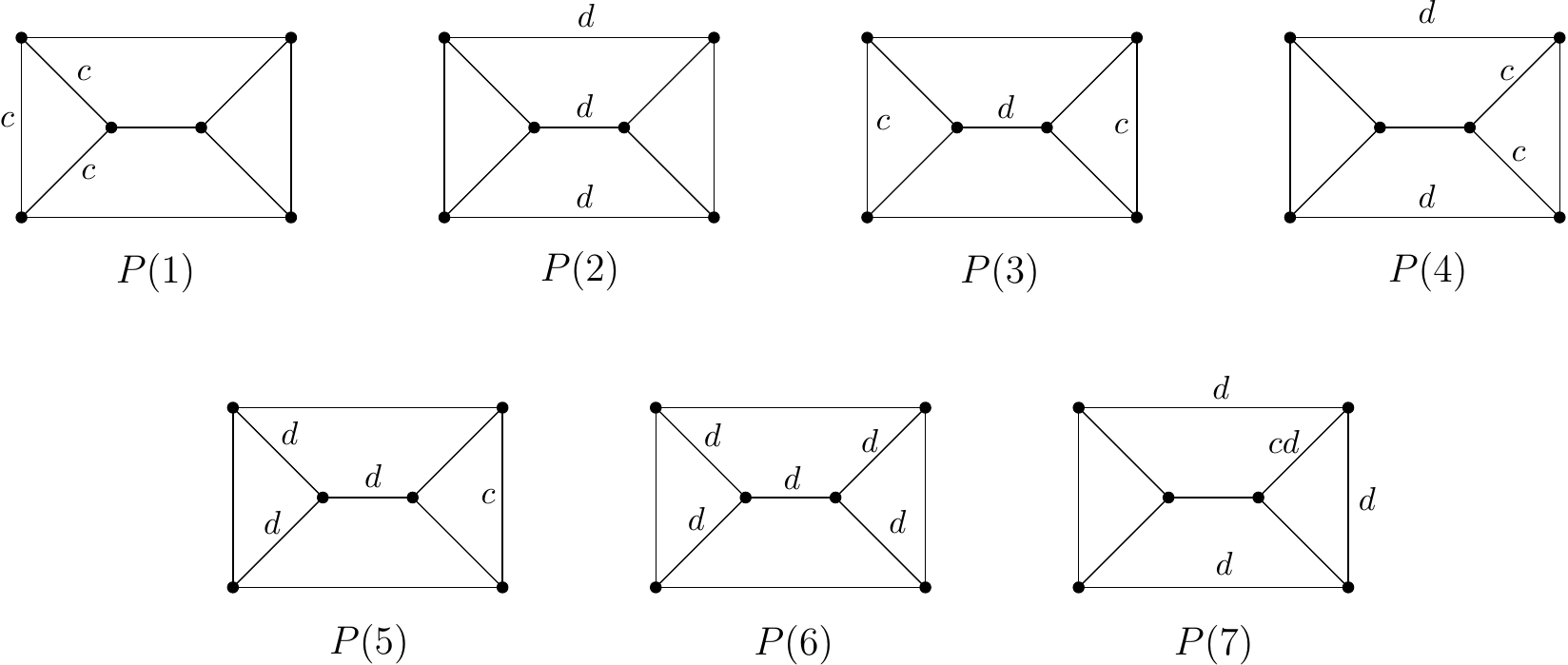}}
\caption{Excluded Minors isomorphic to $P$}
\label{excl_p}
\end {figure}
\end{center}

\begin{lem}
\label{exclp}
The excluded minors with enhanced graph $P$ are $P(1), \ldots, P(7)$.
\end{lem}

\begin{proof} This follows from Lemma \ref{exclk5mlem} and duality (Proposition \ref{enhanced dual}). \end{proof}

\subsection{$P^+$, $D$ and $D^*$}

In this subsection we will first consider non-splitting enhanced graphs isomorphic to $P^+$ appearing in Figure \ref{pp_itself}.  Then we will consider the double fan graph $D$ which is the simple graph obtained from a 3-edge path by adding two new vertices $v$ and $v'$ each adjacent to all vertices of the path but not each other.  This also gives the result for its dual $D^*$.  

$P^+$ is the unique (up to isomorphism) graph which can be obtained from $P$ by adding a new edge.  It is worth noting that $P^+$ has a 2-fold symmetry which fixes both the edges $e$ and $f$ and interchanges the pairs $a_i$ and $b_i$ for $1 \le i \le 4$.  Furthermore, this graph is self-dual, and dualizing reveals that $e^* = f$ and $f^* = e$ while $a_1^* = a_3$ and $a_2^* = a_4$.  %These facts will be helpful in our analysis.

\begin{center}
\begin{figure}[ht]
\centerline{\includegraphics[height=2.7cm]{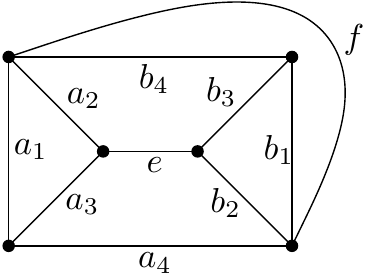}}
\caption{The graph $P^+$}
\label{pp_itself}
\end {figure}
\end{center}

\begin{center}
\begin{figure}[ht]
\centerline{\includegraphics[scale=0.75]{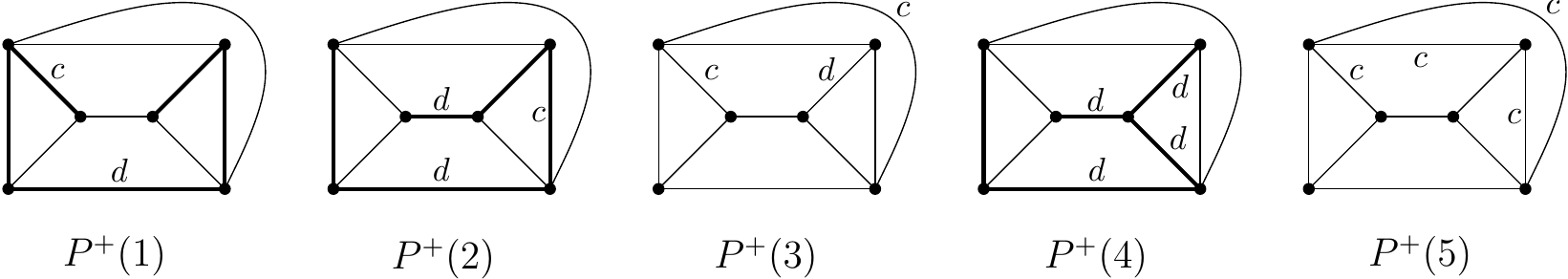}}
\caption{Excluded Minors isomorphic to $P^+$}
\label{excl_pp}
\end {figure}
\end{center}

\begin{lem}
\label{exclpp}
The enhanced graphs $P^+(1), \ldots, P^+(5)$ are minor-minimal non-split
\end{lem}

In fact $P^+(1), \ldots, P^+(5)$ are exactly the excluded minors isomorphic to $P^+$.  This can be shown directly with a case analysis similar but more lengthy to those in the previous subsections.  However, for our purposes it suffices at this stage to simply observe the minimality since the proof of the full characterization in Section \ref{sec full char} will give the rest as a corollary.

\begin{proof}
We may check that these are non-splitting directly from the matroidal definition.

To prove minimality, we must consider smaller enhanced graphs obtained from these by deletion and contraction.  Contracting $f$ results in an enhanced graph of width 2 (which splits) while deleting $f$ results in an enhanced graph isomorphic to Prism which is a proper minor of one of these, and therefore splits.  A similar argument shows that deleting or contracting $e$ yields an enhanced graph which splits.  For every other edge, deletion and contraction result in enhanced graphs which are isomorphic to graphs obtained from either $K_4$ or $W_4$ by series parallel operations, and these can be checked directly.  To see that removing a protection results in a splitting graph, the possibilities can again simply be checked.\end{proof}

%Now consider $D$ and $D^*$.  
%We define the double fan graph, denoted $D$, to be the simple graph obtained from a 3-edge-path by adding two new vertices $v,v'$ each adjacent to all vertices of this path, but not each other.  

\begin{center}
\begin{figure}[ht]
\centerline{\includegraphics[width=7cm]{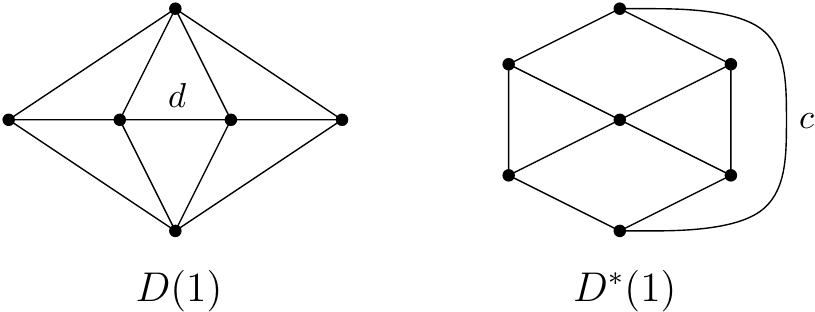}}
\caption{Excluded Minors isomorphic to $D$ and $D^*$}
\end {figure}
\end{center}

\begin{lem} 
\label{excld}
The enhanced graphs $D(1)$ and $D^*(1)$ are minor-minimal non-split.
\end{lem}

\begin{proof}
  By duality it suffices to check $D(1)$.

Check non-splitting directly from the matroidal definition.  For minimality note that $D$ is width 3 and so removing the protection yields a split graph.  Up to symmetry there are 6 other possibilities to check all of which split.
\end{proof}

\section{The full characterization}\label{sec full char}

%\subsection{Non-split graphs have minors in $\mathcal{F}$}

In this section we prove that every non-split enhanced graph contains a minor in $\mathcal{F}$ thus completing the proof of our main result.  
This section is organized as follows. \\

\begin{tabular}{c|p{3.5in}}
Subsection	&	Content \\
\hline
1			&	We begin the proof by choosing a minimal counterexample and finding a special edge partition.  \\
2			&	We establish some basic properties of this partition. \\
3-6			&	We prove that our minimal counterexample cannot have certain particular structures. \\
7			&	We complete the proof of the main theorem.  
\end{tabular}

Our main goal in this section will be to establish the following lemma.

\begin{lemma}
\label{fminor}
If $\widetilde{G}$ is a non-split enhanced graph, then $\widetilde{G}$ contains a minor in $\mathcal{F}$.
\end{lemma}

As the following proof shows, this lemma implies our full excluded minor result.

\bigskip

\noindent{\it Proof of Theorem \ref{mainskeleton}.} It follows from Lemma \ref{fminor} that every non-split enhanced graph contains a minor in 
$\mathcal{F}$.  It follows from Lemmas \ref{k4}, \ref{w4}, \ref{w5}, \ref{exclk5mlem}, \ref{exclp}, \ref{exclpp}, and \ref{excld} that these 
enhanced graphs do not split and no one contains another as a minor.  This completes the proof.
\quad\quad$\Box$

\subsection{Proof Setup}

We begin the proof of Lemma \ref{fminor} by supposing (for a contradiction) that it does not hold.  Choose a minor-minimal enhanced graph $\widetilde{G}$ with a non-split 5-configuration $\widetilde{S}$ which is a counterexample.  The starting point for our analysis is the following corollary.

\begin{center}
\begin{figure}[h]
\centerline{\includegraphics[width=4.5cm]{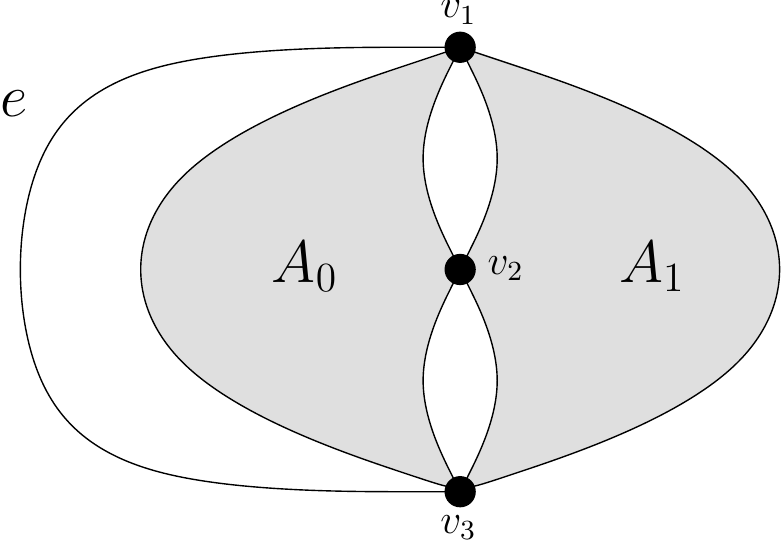}}
\caption{Our starting configuration}
\end {figure}
\end{center}

\begin{corollary}
\label{init_struc}
Either $\widetilde{G}$ or $(\widetilde{G})^*$ has an edge partition $\{ A_0, \{e\}, A_1 \}$ satisfying:
\begin{enumerate}
\item $e \in \widetilde{S}$
\item $|A_i \cap \widetilde{S}| = 2$ for $i=0,1$.
\item $\partial(A_0) = \partial(A_1) = \{v_1,v_2,v_3\}$
\item $e = v_1 v_3$.  
\item At least one of $v_1,v_2,v_3$ has at most one neighbour in $V(A_i) \setminus \{v_1,v_2,v_3\}$ for $i=0,1$.
\end{enumerate}
\end{corollary}

\begin{proof} Use Theorem \ref{decomp} to create an edge ordering $e_1, e_2, \ldots, e_m$ of width 3, and choose $1 \le i \le m$ so that $e_i$ is the third edge of $\widetilde{S}$ in this ordering. Set $A_0 = \{ e_1, \ldots, e_{i-1} \}$, $e = e_i$, and $A_1 = \{ e_{i+1}, \ldots, e_m \}$.  It follows immediately that this partition satisfies the first two properties.  By construction, both $(A_0 \cup \{e\}, A_1)$ and $(A_0, A_1 \cup \{e\})$ are 3-separations in both $\widetilde{G}$ and $(\widetilde{G})^*$, so by possibly interchanging $\widetilde{G}$ with its dual, we may assume that the third and fourth properties holds.  Finally, the fifth is a consequence of width 3.
\end{proof}

Since $\mathcal{F} \setminus \mathcal{F}_0$ is closed under plane duality, by possibly interchanging $\widetilde{G}$ with its dual, we may assume it has an edge partition $\{ A_0,\{e\},A_1 \}$ in accordance with the above corollary.  Furthermore, we may assume that among all such partitions we have chosen $\{A_0,\{e\},A_1\}$ so as to minimize $|A_0|$.  For $i=0,1$ we let $\widetilde{G}_i$ denote the graph induced by $A_i$.

\subsection{Basic Properties}\label{subsec basic}

In this subsection we establish two lemmas which give basic properties of our minimal counterexample $\widetilde{G}$ relative to the partition $\{ A_0, \{e\}, A_1 \}$.  We begin by introducing a structure which will feature prominently in our analysis.

\begin{center}
\begin{figure}[h]
\centerline{\includegraphics[height=4cm]{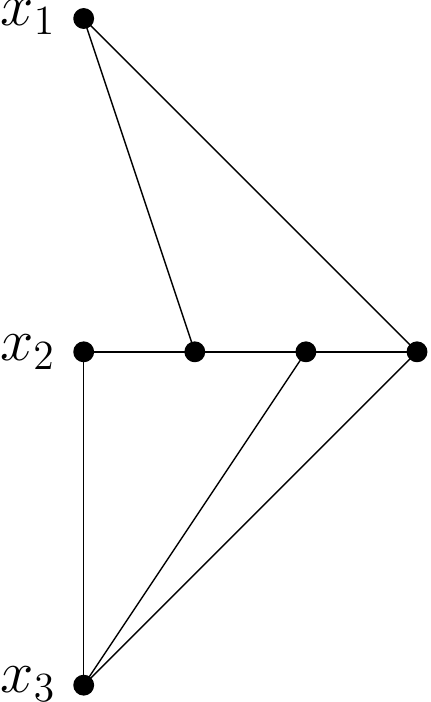}}
\caption{an $(x_1,x_2,x_3)$-doublefan}.  
\end {figure}
\end{center}

Let $G'$ be a graph and let $x_1,x_2,x_3 \in V(G')$ be distinct.  We define $G'$ to be an $(x_1,x_2,x_3)$-\emph{doublefan} 
if $G' \setminus \{x_1,x_3\}$ is a path beginning at $x_2$ and $x_1 x_3$ is not an edge.  We say that $G'$ is a $\{x_1,x_2,x_3\}$-\emph{doublefan} if it is a 
$(x_i,x_j,x_k)$-doublefan for some $i,j,k$ with $\{i,j,k\} = \{1,2,3\}$. If $G'$ is an $(x_1,x_2,x_3)$-doublefan with vertex sequence $w_1, w_2, \ldots, w_n$ for the path $G' \setminus \{x_1,x_2\}$ with $x_2 = w_1$, we call any edge not incident with $x_1$ or $x_3$ a \emph{spine} edge and any edge incident with $x_1$ or $x_3$ a \emph{rib} edge. If needed, we may say an $x_1$-rib or $x_3$-rib, based on incidence.  A rib edge not incident with $w_n$ is called an \emph{inner} rib.  

Now we will define a total preorder on the edges of $G'$.  A preorder is a binary relation which is reflexive and transitive, but it is not required to be anti-symmetric, that is there may be distinct elements $a$ and $b$ with $a\leq b$ and $b\leq a$.  A total preorder is a preorder in which every two elements are comparable.

Returning to the edges of $G'$, if $e,f \in E(G')$ are both rib edges with $e$ incident with $w_i$ and $f$ incident with $w_j$ then we define $e \le f$ if $i \le j$.  If $e,f$ are both spine edges, say $e = w_i w_{i+1}$ and $f = w_j w_{j+1}$ then we define $e \le f$ if $i \le j$.  Finally, if $e = w_i w_{i+1}$ is a spine edge and $f$ is a rib edge incident with $w_j$ then we put $e \le f$ if $i < j$ and $f \le e$ otherwise.  

Note that this is a total preorder rather than a total order because a given spine vertex may have two ribs coming off of it both of which occupy the same place in the order.

\begin{lemma} \label{spinerib} 
Suppose $G_i$ is a $(v_1,v_2,v_3)$-doublefan and $f\leq f'$ with $f,f' \in A_i \cap S$.  Then if $f$ is a rib it must be contract-proof and if $f$ is a spine it must be delete-proof.
%Suppose for a graph non-split with non-splitting 5-configuration $S$ $A_0$, $e$, and $A_1$ are as constructed in Corollary \ref{init_struc} with $e_1, e_2 \in A_0$ and $e_4,e_5 \in A_1$ where $S = \{e_1,e_2,e_3,e_4,e_5\}$. In $A_i$ for $i = 0,1$, the edge in $S$ of lower order must be protected; delete-proof if it is a spine and contract-proof otherwise.  \iain{horrid} 
\end{lemma} 

\begin{proof} This is immediate, as otherwise vertices $v_1,v_3$ create a 2-separation that partitions the edges of $S$ into $\{e,f'\}$ and $S \smallsetminus \{e,f,f'\}$.  \end{proof}

\begin{lemma}
\label{basic_decomp}
For both $i=0,1$ one of the following holds:
\begin{enumerate}
\item $E(\widetilde{G}_i) = \{ v_1 v_2, v_2 v_3 \}$.
\item $\widetilde{G}_i \setminus \{v_1,v_2,v_3\}$ contains a cycle.
\item $\widetilde{G}_i$ is a $\{v_1,v_2,v_3\}$-doublefan.
\end{enumerate}
\end{lemma}

\begin{proof} Since $\widetilde{G}$ is simple $v_1v_3 \not\in \widetilde{G}_i$ and since $|S\cap A_i| = 2$, $|A_i| \geq 2$.  

We shall assume that the first and second properties do not hold, and show the third does.  By these assumptions, the graph $\widetilde{G}_i  \setminus \{v_1,v_2,v_3\}$ is a nonempty forest.  Suppose first that $\widetilde{G}_i \setminus \{v_1,v_2,v_3\}$ has at least two components.  If $V(\widetilde{G}_{1-i}) \neq \{v_1,v_2,v_3\}$, then $\widetilde{G} \setminus \{v_1,v_2,v_3\}$ has at least three components, and then by 3-connectivity $\widetilde{G}$ contains 
a $K_{3,3}$-minor which is contradictory.  So, we must have $V(G_{1-i}) = \{v_1,v_2,v_3\}$ and since $|A_{1-i} \cap S| = 2$ and $e = v_1 v_3$ it can only be that 
$A_{1-i} = \{ v_1 v_2, v_2 v_3 \} \subseteq \widetilde{S}$.  As the edges in $A_{1-i} \cup \{e\}$ form a triangle, and since each of these edges is in $\widetilde{S}$ they must all be contract-proof, and it then follows that $\widetilde{G}$ contains $K_5^-(2)$ as a minor, a contradiction.  Thus $\widetilde{G}_i \setminus \{v_1,v_2,v_3\}$ must be connected, so it is a tree.

\begin{center}
\begin{figure}[h]
\centerline{\includegraphics[width=2cm]{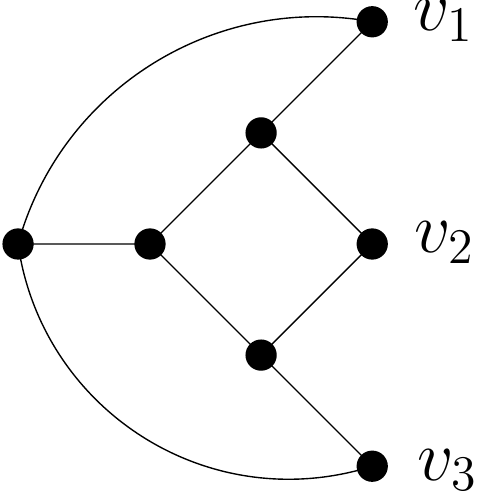}}
\caption{}
\label{cube-v}
\end {figure}
\end{center}

If $\widetilde{G}_i \setminus \{v_1,v_2,v_3\}$ has at least three leaves, then by 3-connectivity and planarity it must contain the graph from Figure \ref{cube-v} as a rooted minor.  However, $\widetilde{G}_{1-i} \cup e$ must contain a rooted minor consisting of a triangle on $v_1,v_2,v_3$ or a $K_{1,3}$ with degree one vertices $v_1,v_2,v_3$, and thus must $\widetilde{G}$ contains the cube or $H$ as a minor, which is contradictory.  So, it must be that $\widetilde{G}_i \setminus \{v_1,v_2,v_3\}$ is a path.  By 3-connectivity and planarity, each of $v_1,v_2,v_3$ must be adjacent to at least one endpoint of this path.  It now follows from the last property in Corollary \ref{init_struc} that $\widetilde{G}_i$ is a $\{v_1,v_2,v_3\}$-doublefan.
\end{proof}

\medskip

Our next lemma exploits the assumption that we have chosen our partition $(A_0,\{e\},A_1)$ with $|A_0|$ minimum.  

\begin{lemma}
\label{iftriangle}
If there is a triangle in $\widetilde{S}$, then $A_0 \cup \{e\}$ is a triangle of edges in $\widetilde{S}$.
\end{lemma}

\begin{proof}
Suppose that $\widetilde{S}$ contains three edges which form a triangle, and let $A$ denote a set of  three such edges.  Note that by Lemma \ref{mm3sep} all edges in $A$ must be contract-proof.  Let $B$ be the complement of $A$ and let $\partial(B) = \{u_1,u_2,u_3\}$.  Suppose for a contradiction that each of $u_1,u_2,u_3$ has at least two neighbours in the set $B \setminus \{u_1,u_2,u_3\}$.  If $\widetilde{G} \setminus \{u_1,u_2,u_3\}$ has at least two components, then by 3-connectivity it contains $K_5^-(2)$ as a minor, which is a contradiction.  If $\widetilde{G} \setminus \{u_1,u_2,u_3\}$ contains a cycle, then by 3-connectivity, it contains $P(1)$ as a minor, which is a contradiction.  Therefore, $\widetilde{G} \setminus \{u_1,u_2,u_3\}$ is a tree.  As in the previous lemma, if $\widetilde{G} \setminus \{u_1,u_2,u_3\}$ has at least three leaves, then by 3-connectivity $\widetilde{G} \setminus A$ contains the rooted minor in Figure \ref{cube-v}.  However, then $\widetilde{G}$ contains an $H$ minor which is a contradiction.  So, it must be that $\widetilde{G} \setminus \{u_1,u_2,u_3\}$ is a path, and we let $w_1, w_2, \ldots, w_n$ be the vertex sequence of this path.  Since $\widetilde{G}$ is 3-connected, we may assume (by possibly reindexing) that $w_1$ is adjacent to $u_1$ and $u_2$ and $w_n$ is adjacent to $u_2$ and $u_3$ as shown on the left in Figure \ref{tame_triangle}.

\begin{center}
\begin{figure}[h]
\centerline{\includegraphics[width=10cm]{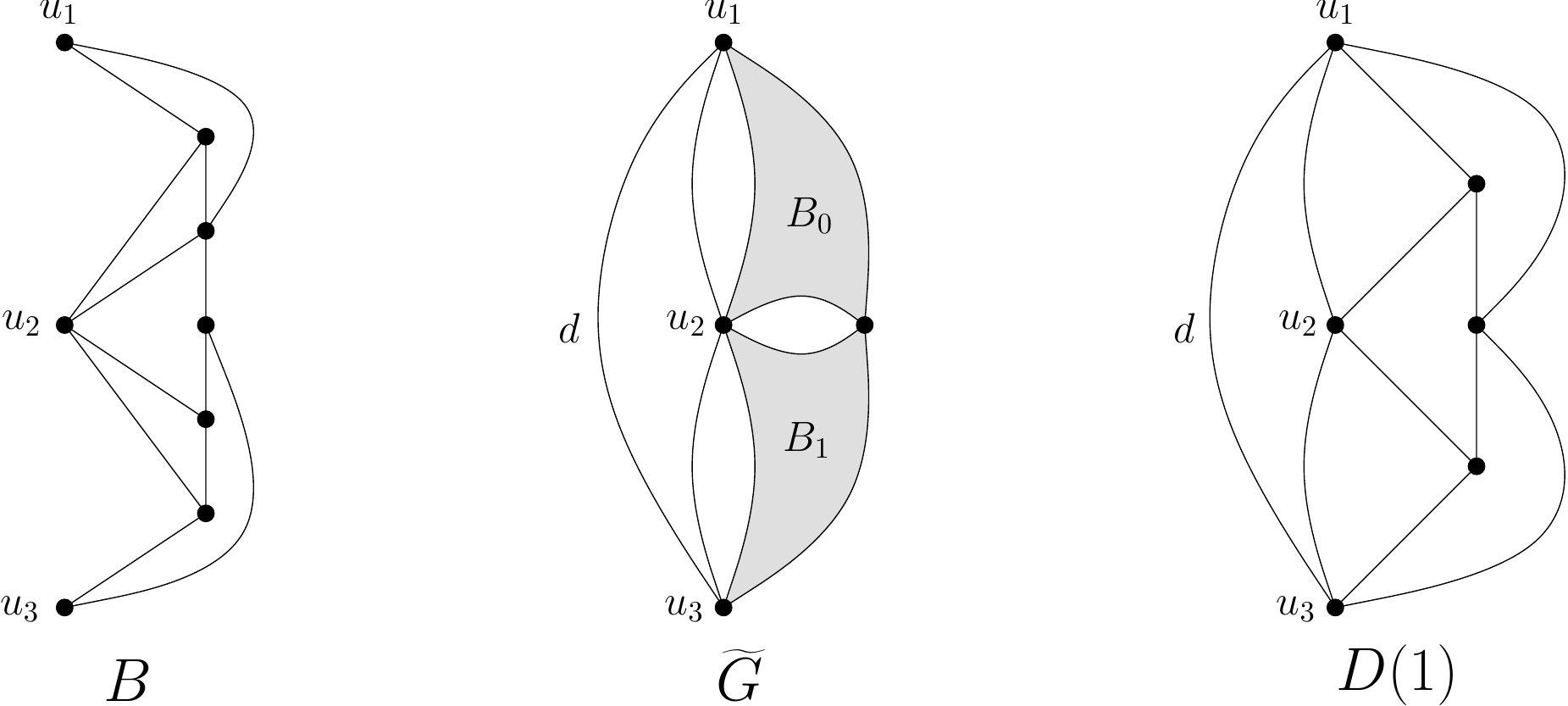}}
\caption{}
\label{tame_triangle}
\end {figure}
\end{center}

Let $1 \le j \le n$ be the highest index for which $w_j$ is adjacent to $u_1$, let $B_0$ consist of all edges in the subpath between $w_1$ and $w_j$ together with all edges between $\{u_1,u_2\}$ and $\{ w_1, \ldots, w_j \}$ and set $B_1 = B \setminus B_0$.  Then by planarity our entire graph $\widetilde{G}$ has the structure shown in the second graph of Figure \ref{tame_triangle}.  Now, by our assumptions, each of $u_1$ and $u_3$ has at least two neighbours in $\{w_1, \ldots, w_n\}$ and it follows that $|B_0|, |B_1| \ge 4$.  Therefore, Lemma \ref{mm3sep} implies that $|B_0 \cap \widetilde{S}| = 1$ and $|B_1 \cap \widetilde{S}| = 1$.  It follows from this and the assumption that $\widetilde{S}$ does not split that the edge $u_1 u_3$ must be delete-proof.  However, then $\widetilde{G}$ contains $D(1)$ as a minor (as shown in the last graph of Figure \ref{tame_triangle}) which is a contradiction.  

Therefore, one of the vertices $u_1,u_2,u_3$ must have at most one neighbour in $B \setminus \{u_1,u_2,u_3\}$.  Now partitioning $A$ into a set 
$A'$ of size two and a singleton $\{e'\}$ we find that $( A', \{e'\},  B )$ satisfies the properties in Corollary \ref{init_struc}.  Now, since we have chosen our partition 
$\{ A_0, \{e\}, A_1\}$ with $|A_0|$ minimum, it follows that $|A_0| = 2$ and then $A_0 \cup \{e\}$ is a triangle, as desired.
\end{proof}

\subsection{A triangle in the 5-configuration}\label{subsec triangle}

For convenience in the remainder of this section, let us set $e_3 = e$ and let $\{e_1,e_2\} = A_0 \cap \widetilde{S}$ and $\{e_4,e_5\} = A_1 \cap \widetilde{S}$.  
The purpose of  this subsection is to establish the following lemma.

\begin{lemma}
\label{notriangle}
There does not exist a triangle in $\widetilde{S}$.
\end{lemma}

\begin{proof} By Lemma \ref{iftriangle} we may assume that the set $A_0^+ = A_0 \cup \{e\} = \{e_1,e_2,e_3\}$ is a triangle of edges all contained in $\widetilde{S}$.  If $\widetilde{G} \setminus \{v_1,v_2,v_3\}$ contains a cycle, then by 3-connectivity $\widetilde{G}$ contains $P(1)$ 
as a minor, which is contradictory.  It then follows from Lemma \ref{basic_decomp} that $\widetilde{G}_1$ is a $\{v_1,v_2,v_3\}$-doublefan.  By possibly switching $e_3$ with another edge in $A_0^+$ and relabeling, we may then assume that $\widetilde{G}_1$ is a $(v_1,v_2,v_3)$-doublefan.  Let $w_1, w_2, \ldots, w_n$ be the vertex sequence of the path $\widetilde{G} \setminus \{v_1,v_3\}$ and assume $v_2 = w_1$.  If $\widetilde{G}$ is a wheel, then it follows from results in Section \ref{secwheel} that it must have an excluded minor in $\mathcal{F}$ which is contradictory.  Therefore each of $v_1$ and $v_3$ must have a neighbour in the set $\{w_2, \ldots, w_{n-1} \}$.  By possibly interchanging $v_1$ and $v_3$, we may assume that $v_1$ is incident with $w_2$ and we now choose 
the smallest $2 \le j \le n$ for which $v_3 w_j$ is an edge.  Note that we now have a structure similar to that in Figure  \ref{triangle-doublefan}.

\begin{center}
\begin{figure}[h]
\centerline{\includegraphics[height=4cm]{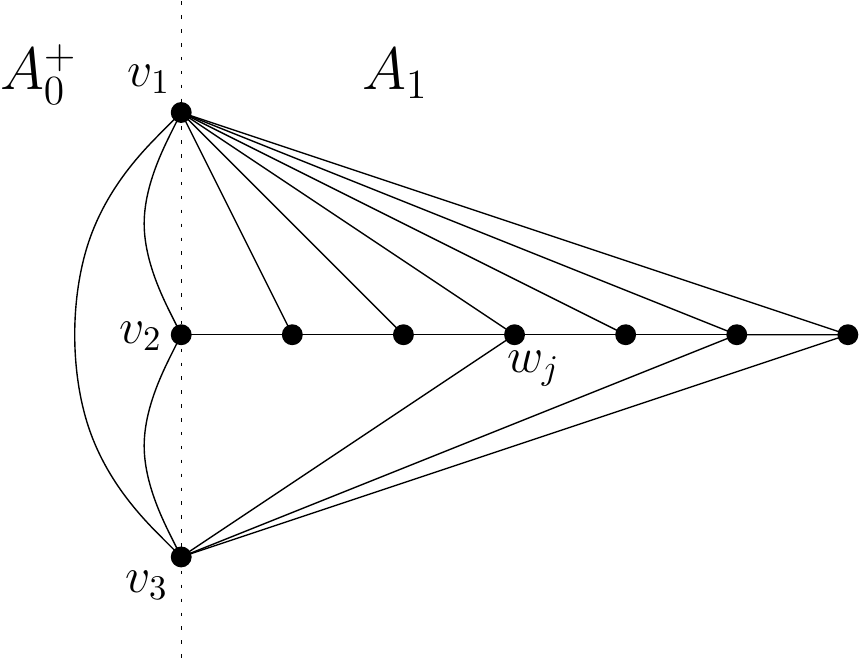}}
\caption{}
\label{triangle-doublefan}
\end {figure}
\end{center}

Assume without loss that $e_4 \le e_5$.  We now consider possibilities for the edge $e_4$.  

First suppose that $e_4$ is a spine edge.  In this case, it follows from Lemma \ref{spinerib} that $e_4$ must be delete-proof.  If $e_4 = w_1 w_2$, then $v_2 = w_1$ is a vertex of degree three incident with three edges of $\widetilde{S}$, so all must be delete-proof, and we find that $\widetilde{G}$ contains $K_5^-(1)$ as a minor (as seen in the first graph of Figure \ref{notriangle_extras}), which is a contradiction.  If $e_4 = w_i w_{i+1}$ where $1 < i < j$, then the edge $v_2 v_3$ must be delete-proof, as otherwise $G \setminus v_2 v_3$ has a bad 2-separation using the vertices $v_1, w_{i+1}$.  However, in this case $\widetilde{G}$ contains a $W_4(5)$ minor (as shown in the second graph of Figure \ref{notriangle_extras}), which is a contradiction.  In the remaining case $e_4 = w_i w_{i+1}$ where $j \le i$ and $\widetilde{G}$ contains a $K_5^-(5)$ minor (as shown in the third graph of Figure \ref{notriangle_extras}), which is contradictory.

\begin{center}
\begin{figure}[h]
\centerline{\includegraphics[scale=0.65]{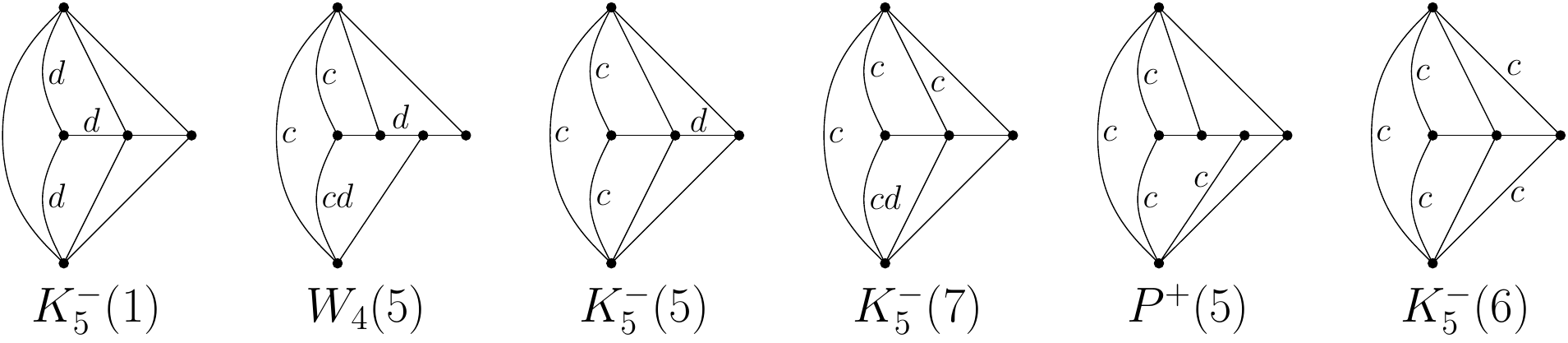}}
\caption{}
\label{notriangle_extras}
\end {figure}
\end{center}

Next we suppose that $e_4$ is a rib edge.  If $e_4 = v_1 w_i$ for $i \le j$ then the edge $v_2 v_3$ must be delete-proof as otherwise $\widetilde{G} \setminus v_2 v_3$ has a bad 2-separation using the vertices $v_1$ and $w_i$.  In this case $\widetilde{G}$ has a $K_5^-(7)$ minor (seen in the fourth graph in Figure \ref{notriangle_extras}) which 
is contradictory.  Similarly, if $e_4 = v_3 w_2$ then $v_1 v_2$ is delete-proof and we have a $K_5^-(7)$ minor which is contradictory.  It follows from this that if 
$e_4 = v_1 w_i$ then there is a rib edge $v_3 w_k$ for $1 < k < i$ and if $e_4 = v_3 w_i$ then there is a rib edge $v_1 w_k$ for $1 < k < i$.  If $e_4$ is an inner rib, then we have $P^+(5)$ as a minor (as in the fifth graph of Figure \ref{notriangle_extras}) which is a contradiction.  Otherwise, $e_4$ is incident with $w_n$, and by our choice $\{e_4,e_5\}$ must equal $\{v_1 w_n, v_3 w_n \}$ and both edges must be contract-proof giving us a $K_5^-(6)$ minor (as in the last graph of Figure \ref{notriangle_extras}) and a final contradiction. \end{proof}

\subsection{Two central doublefans}

\begin{lemma}
\label{central_doublefan}
$\widetilde{G}_0$ and $\widetilde{G}_1$ are not both $(v_1,v_2,v_3)$-doublefans.
\end{lemma}

\begin{proof}
We break this proof into cases, considering the edges in $\widetilde{S}$ of lowest order in both $G_0$ and $G_1$. Suppose without loss that $e_1, e_2 \in E(\widetilde{G}_0)$ and $e_4,e_5 \in E(\widetilde{G}_1)$, $e_1 \leq e_2$ and $e_4 \leq e_5$ in their respective doublefan preorders. By Lemma \ref{general_doublefan}, this determines some of the needed protection.

\startingcases
\begin{case} Both $e_1$ and $e_4$ are spine edges. \end{case}

If there exist both $v_1$- and $v_3$-ribs of lower order, we immediately produce a $K_5^-(3)$ minor. Hence we may assume by symmetry that there are only $v_3$-ribs between $e_1$ and $e_4$. If there is at least one inner $v_1$-rib edge in both $G_0$ and $G_1$ then there will necessarily be a $W_4(2)$ minor. Thus, without loss we may restrict to the case that $G_0$ has no inner $v_1$-ribs. 

In $G_0$, edge $e_2$ must necessarily be protected. If there are no inner $v_1$-ribs, then $\widetilde{G}$ is a wheel and hence has a minor in $\mathcal{F}$. We may therefore assume that there is an inner $v_1$-rib in $G_1$. In any case, this graph must have a minor isomorphic to one of the graphs in Figure \ref{spinespinecases}; the first if $e_2$ is a spine or $v_1$-rib, the second otherwise. 

\begin{figure}[h]
  \centering
    \includegraphics[scale=0.65]{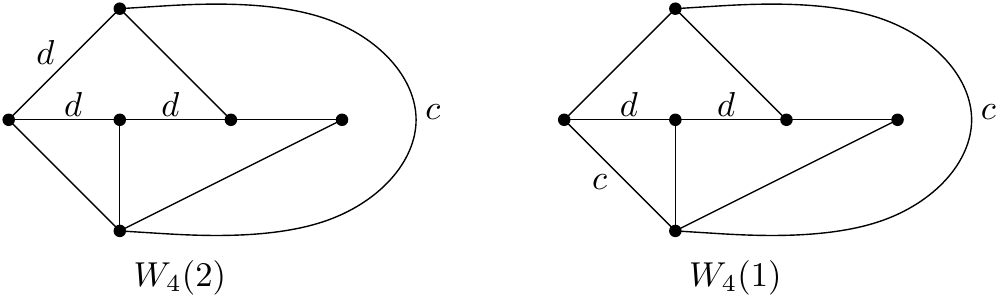}
  \caption{}
\label{spinespinecases}
\end{figure}

\begin{case} One of $e_1,e_4$ is a spine, the other a rib. \end{case}

Without loss of generality, assume that $e_4$ is a $v_3$-rib in $G_1$.  Further, we may assume that $e_4$ is an inner rib, as otherwise this reduces to Lemma \ref{iftriangle}.

If $e_1$ and $e_4$ do not share a vertex and there is a strictly lower order $v_1$-rib that does not share a vertex with $e_4$, then there is a $P^+(3)$ minor. Hence, either $e_1$ and $e_4$ share a vertex or all rib edges of lower order are $v_3$-ribs.

If there is inner $v_1$-rib in both $G_0$ and $G_1$ we get a $K_5^-(4)$ minor, seen in the fourth graph of Figure \ref{spineribcases}. If neither side has an inner $v_1$-rib, $\widetilde{G}$ is a wheel.

If $e_1$ is a spine furthest along this path from $v_2$, then $e_2$ must be protected; if $e_2$ is a $v_1$-rib it must be delete-proof, and otherwise it must be contract-proof and $e_1$ must be double protected. In either case, we may assume that there is a non-terminal $v_1$-rib in $G_1$ to avoid a wheel enhanced graph. This produces either a $K_5^-(4)$ or $K_5^-(7)$ minor. 

If we suppose that $e_1$ is not a spine of highest index, then precisely one of $G_0$ or $G_1$ must have at least one inner $v_1$-rib, and either $G_0$ or $G_1$, call it $G_i$, must contain only $v_3$-ribs. As such the second edge in $\widetilde{S}$ in $G_i$ must be protected; contract-proof if it is a $v_3$-rib and delete-proof otherwise. In all cases, this produces a non-splitting graph; each shown in Figure \ref{spineribcases}. The first row if $G_1$ has no $v_1$-ribs, and the second row if $G_0$ has no $v_1$-ribs. Columns correspond to spine, $v_1$-rib, and $v_3$-rib protection, respectively. Note in particular that the first $W_4(6)$ minor is produced if the two selected spine edges share a vertex or not, since there may only be $v_3$-ribs between if the two distinguished edges do not share a vertex, and the rib must be double protected otherwise.  Furthermore, in the case of a $v_3$-rib, illustrated in the last graph of Figure \ref{spineribcases}, we can ignore the case when this $v_3$-rib shares a vertex with $e_1$ by Lemma \ref{notriangle}.

\begin{figure}[h]
  \centering
    \includegraphics[scale=0.65]{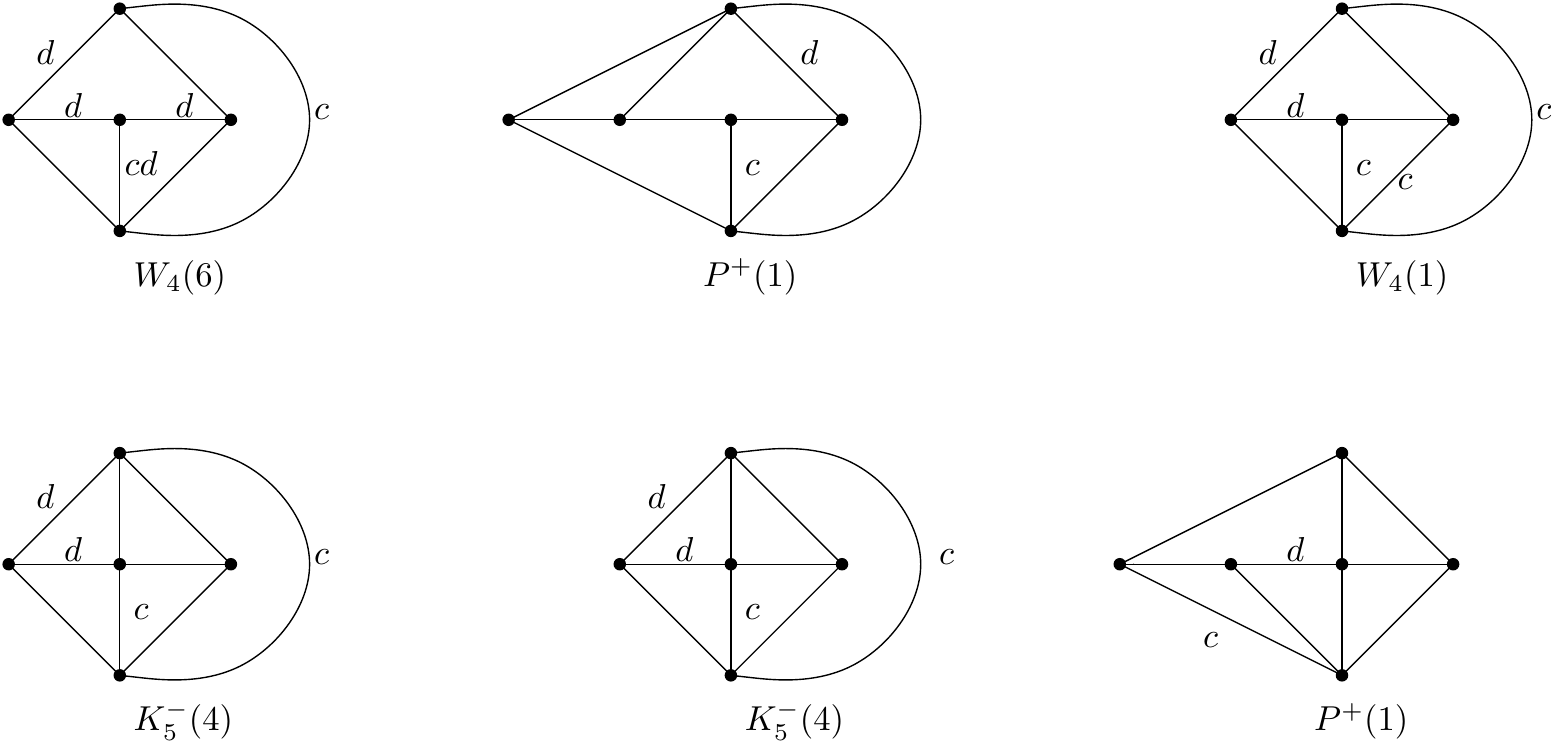}
  \caption{}
\label{spineribcases}
\end{figure}

\begin{case} Both $e_1$ and $e_4$ are ribs. \end{case}

Again, we may assume that $e_1$ and $e_4$ are inner ribs by Lemma \ref{notriangle}. By symmetry both distinguished ribs may be assumed to be $v_3$-ribs, as otherwise we get a $K_5^-(2)$ minor. If there exists a $v_1$-rib of strictly lower order, we get a $P^+(5)$ minor, the first graph in Figure \ref{ribribcases}. Thus, we may assume that either there are no vertices properly between $e_1$ and $e_4$ or that there are only $v_3$-ribs properly between. In either case, there may be at most one side with an inner $v_1$-rib after (or incident to) these ribs, as otherwise we get a $K_5^-(7)$ minor, the second graph in Figure \ref{ribribcases}.

If there are no vertices or only $v_3$-ribs between the two closest ribs, then precisely one of $G_0$ or $G_1$ must have at least one inner $v_1$-rib to avoid $\widetilde{G}$ being a wheel; say this is $G_i$. The second edge in $\widetilde{S}$ in $G_i$ must be protected; contract-proof if it's a $v_3$-rib and delete-proof otherwise. These are all non-splitting; the last three graphs in Figure \ref{ribribcases} shows these non-splitting minors.

\begin{figure}[h]
  \centering
    \includegraphics[scale=0.7]{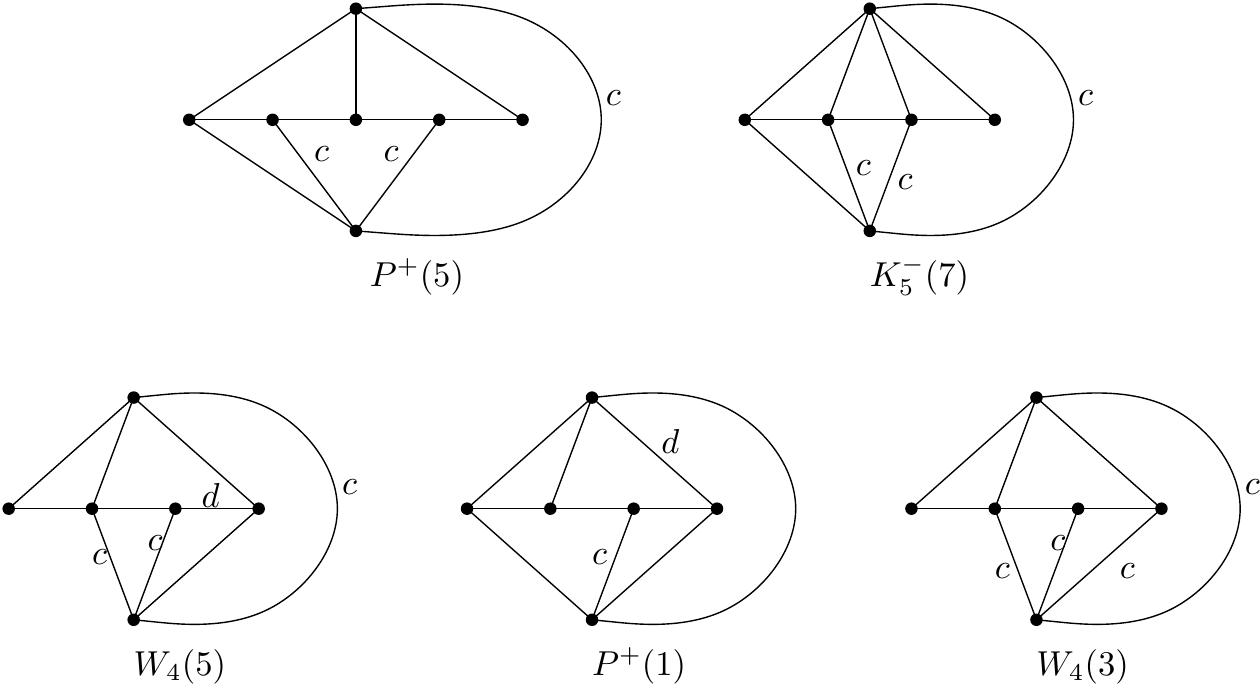}
  \caption{}
\label{ribribcases}
\end{figure}

Hence, any minor-minimal non-split graph with two $(v_1,v_2,v_3)$-doublefans is already in $\mathcal{F}$. \end{proof}

\subsection{Two doublefans}

\begin{lemma}
\label{general_doublefan}
If $\widetilde{G}_i$ is a $(v_1,v_2,v_3)$-doublefan with no delete-proof spine edge, then $\widetilde{G}_{1-i}$ is not a $\{v_1,v_2,v_3\}$-doublefan.
\end{lemma}

\begin{proof}
  Suppose $\widetilde{G}_0$ is a $\{v_1 , v_2 , v_3\}$-doublefan and $\widetilde{G}_1$ is a $(v_1 , v_2 , v_3 )$-doublefan with no delete-proof spine edge.

  Again, $e_1,e_2 \in \widetilde{G}_0$ with $e_1 \geq e_2$ and  $e_4,e_5\in \widetilde{G}_1$ with $e_4 \leq e_5$. As $e_4$ must be protected, by assumption it cannot be a spine edge.  Without loss of generality say $e_4$ is a $v_1$-rib.
  By Lemma \ref{central_doublefan} we can assume that $\widetilde{G}_0$ is either a $(v_1, v_3, v_2)$-doublefan or a $(v_2,v_1,v_3)$-doublefan.  By Lemma \ref{notriangle} we know that $e_4$ must be an inner rib.

\startingcases
\begin{case} Edge $e_4$ is not incident to $v_2$.  \end{case}
 By 3-connectivity of $\widetilde{G}$, there is at least one additional $v_2$-rib. 
If there is a $v_2$-rib which is not incident to $v_1$ or $v_3$ then $G$ has a $P(1)$ minor.  Otherwise then $\widetilde{G}$ consists of two central doublefans and we are done by Lemma \ref{central_doublefan}.  These possibilities are illustrated in Figure \ref{dfdf_case1}.

\begin{figure}[h]
  \centering
    \includegraphics{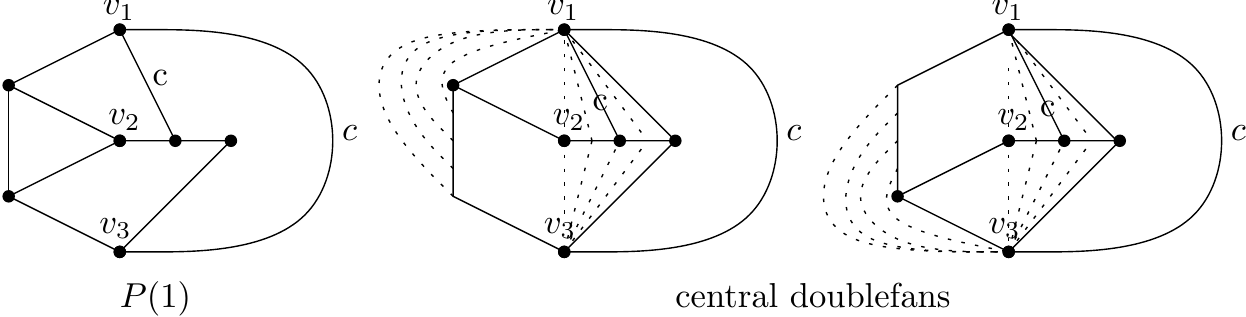}
  \caption{}
\label{dfdf_case1}
\end{figure}

\begin{case} $\widetilde{G}_0$ is a $(v_1,v_3,v_2)$-doublefan and $e_4$ is incident to $v_2$. \end{case}
  In this case $e_2$ must be protected.  Let $v$ be the last vertex in the spine of $\widetilde{G}_0$.  This case has subcases based on what kind of edge $e_2$ is.

Suppose $e_2$ is a $v_2$-rib.  Then $e_2$ is contract-protected.  If it is not incident to $v$ then $G$ has a $K_5^-(2)$ minor as in Figure \ref{dfdf_case2a}.  Suppose $e_2$ is incident to $v$. Then $e_1$ is the rib from $v$ to $v_1$ and so we are done by Lemma \ref{notriangle}.
%is also protected.  If there is another rib to $v_2$ in $\widetilde{G}_0$, but not to $v_3$, then we get a $P^+(5)$ minor.  If not then $\widetilde{G}_0$ is a central doublefan and so we are done by Lemma \ref{central_doublefan}.  These possibilites are illustrated in Figure \ref{dfdf_case2a}

\begin{figure}[h]
  \centering
    \includegraphics{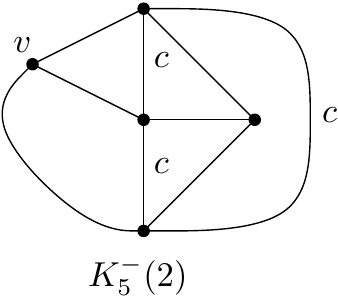}
  \caption{}
\label{dfdf_case2a}
\end{figure}

Suppose $e_2$ is a $v_1$-rib.  We may assume that $e_2$ is not incident to $v$ by Lemma \ref{notriangle}.  Suppose $e_1$ is not incident to $v$.  If there are no inner $v_2$-ribs in $\widetilde{G}_0$ and no inner $v_3$-ribs in $\widetilde{G}_1$ then $\widetilde{G}$ is a wheel and so by Lemma \ref{nonsplit_wheel} it has one of the known minors.  Assume there is at least one such rib $r$.  There are now two possibilities. 

\begin{figure}[h]
  \centering
    \includegraphics[scale=0.72]{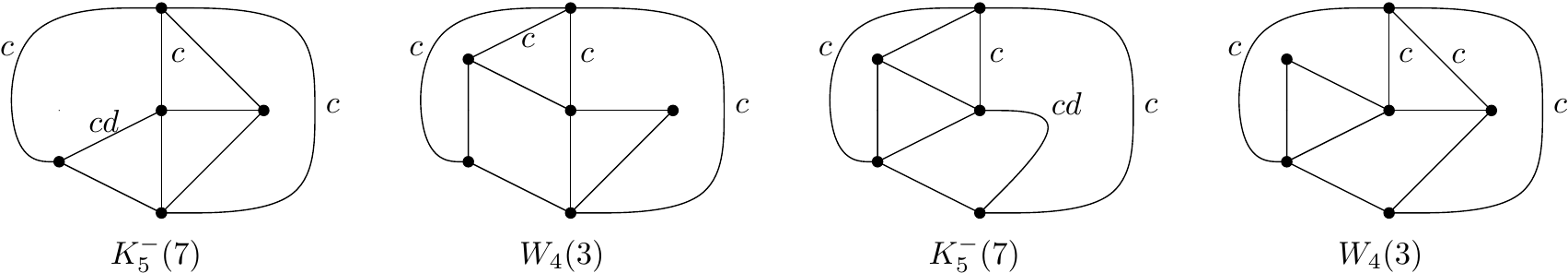}
  \caption{}
\label{dfdf_case2b}
\end{figure}

\begin{itemize}
  \item Suppose $r$ is either in $\widetilde{G}_1$ or is before $e_2$ in $\widetilde{G}_0$.  To avoid a 2-separation separating $e_2$ and $e_3$ from $e_4$ and $e_5$ then either there is a $v_2$-rib in $\widetilde{G}_0$ at or after $e_2$, or $e_1$ is protected.  In the case of the extra $v_2$-rib or a delete-protected $e_1$, i.e. $e_1$ on the spine or $e_1$ the rib from $v$ to $v_2$, then $G$ has a $K_5^-(7)$ minor.  In the case where $e_1$ is a contract-protected $v_1$-rib we get a $W_4(3)$ minor.  These possibilities are illustrated in the first two graphs of Figure \ref{dfdf_case2b}.
  \item Suppose $r$ is at or above $e_2$ in $\widetilde{G}_0$ and we are not in the previous point.  As above but with the sides flipped, to avoid a 2-separation separating $e_2$ and $e_3$ from $e_1$ and $e_4$ then $e_5$ must be protected.  If $e_5$ is delete protected then $G$ has a $K_5^-(7)$ and otherwise $G$ has a $W_4(3)$ minor.  These possibilities are illustrated in the third and fourth graphs of Figure \ref{dfdf_case2b}.
\end{itemize}

To complete Case 2 suppose $e_2$ is a spine.  If there is no inner $v_2$-rib in $\widetilde{G}_0$ except possibly from $v_2$ to $v_3$ then we are done by Lemma \ref{central_doublefan}.  So assume there is such a $v_2$-rib.  If this $v_2$-rib is smaller than $e_2$ then we get a $P^+(3)$ minor as illustrated in the first graph of Figure \ref{dfdf_case2_spine}.  Assume all $v_2$-ribs are larger than $e_2$ and there is at least one such rib.
 To avoid a 2-cut separating $e_4$ and $e_1$ from $e_2$ and $e_3$, either $e_5$ is protected or there is an inner $v_3$-rib.  There are four possibilities

\begin{figure}[h]
  \centering
    \includegraphics{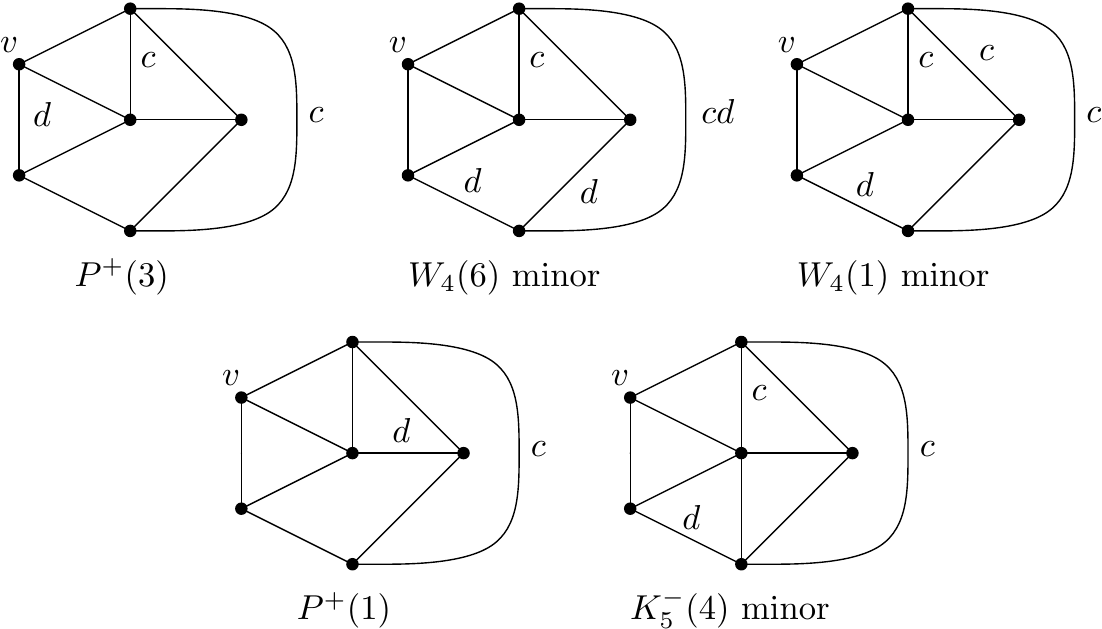}
  \caption{}
\label{dfdf_case2_spine}
\end{figure}
%\iain{Should we switch these to the actual graphs rather than the minors?}

\begin{itemize}
\item If $e_5$ is protected and is a $v_3$-rib (possibly from $v_2$ to $v_3$) and there is no inner $v_3$-rib,  then either there is also a $v_1$-rib in $\widetilde{G}_0$ before $e_2$, or $e_3$ is double protected.  In either case this gives $W_4(6)$ as a minor.  
\item  If $e_5$ is protected and is a $v_1$-rib then $G$ has a $W_4(1)$ minor.
\item If $e_5$ is protected and is a spine then $G$ has a $P^+(1)$ minor.
\item If there is an inner $v_3$-rib then $G$ has a $K^-_5(4)$ minor.
\end{itemize}
  These possibilities are illustrated in the last four graphs of Figure \ref{dfdf_case2_spine}

\begin{case} $\widetilde{G}_0$ is a $(v_2,v_1,v_3)$-doublefan and $e_4$ is incident to $v_2$. \end{case}
 Assume we are not in the previous case, and so there is at least one inner $v_3$-rib in $\widetilde{G}_0$.
If there is at least one inner $v_2$-rib then $G$ has a $K_5^-(2)$ minor.  Otherwise again both sides become central doublefans so appeal to Lemma \ref{central_doublefan}.  These possibilities are illustrated in Figure \ref{dfdf_case3}.

\begin{figure}[h]
  \centering
    \includegraphics{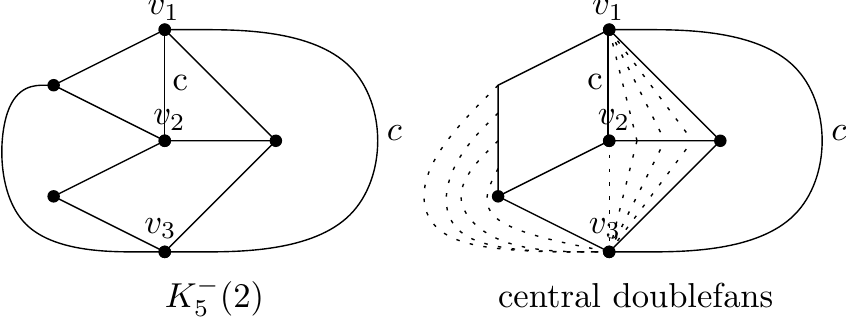}
  \caption{}
\label{dfdf_case3}
\end{figure}

\end{proof}

\subsection{Doublefan \& ladder}\label{subsec df-l}

In this subsection we establish another special case of the main theorem by proving that our minimal counterexample $\widetilde{G}$ cannot 
have a particular structure.  However, we first need to define a key structure of interest.

\begin{center}
\begin{figure}[h]
\centerline{\includegraphics[height=5.5cm]{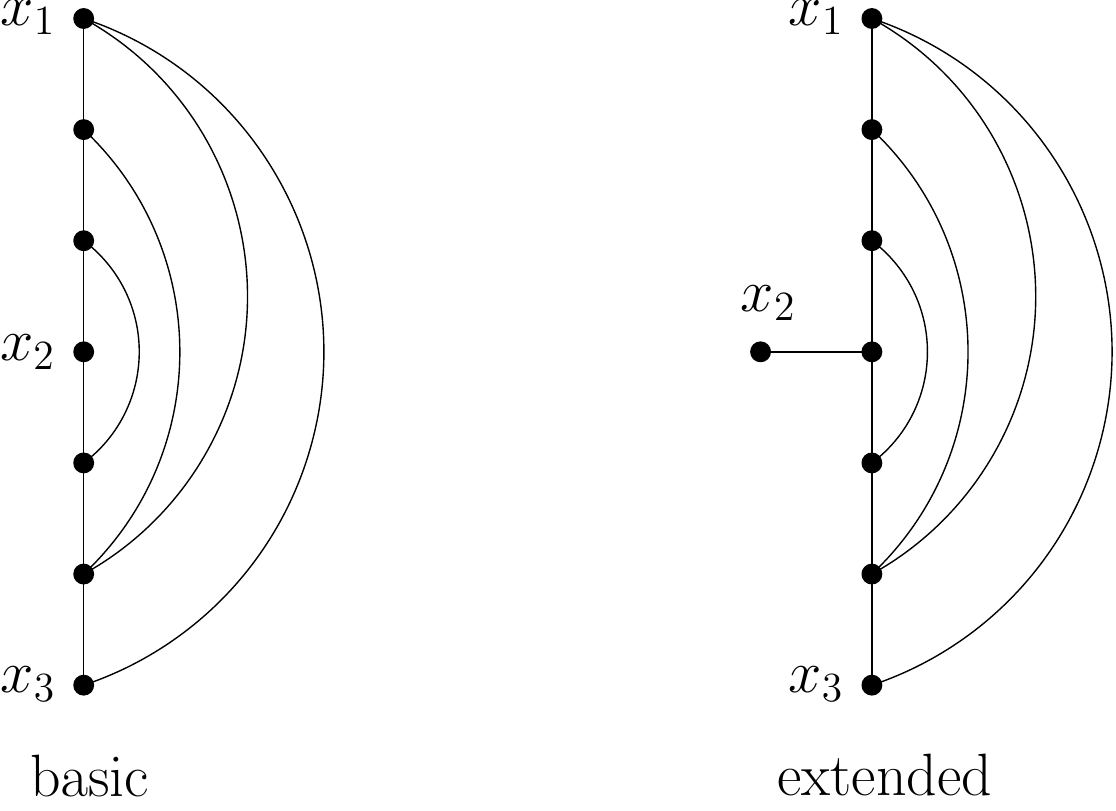}}
\caption{$(x_1,x_2,x_3)$-ladders}
\end {figure}
\end{center}

Let $L$ be a graph with distinct vertices $x_1,x_2,x_3$.  We define $L$ to be a \emph{basic} $(x_1,x_2,x_3)$-\emph{ladder} 
If there exist two internally disjoint paths $P_1,P_2$ in $L$ with the following properties:
\begin{enumerate}
\item $P_i$ is a path of length at least 2 from $x_i$ to $x_{i+1}$ for $i=1,2$.
\item Every vertex apart from $x_1,x_2,x_3$ has degree at least 3.
\item Every edge in $E(L) \setminus E(P_1 \cup P_2)$ has one end in $P_1$ and the other in $P_2$.
\item There do not exist edges $y_1y_2, z_1 z_2 \in E(L) \setminus E(P_1 \cup P_2)$ which ``cross'' in the sense that 
$y_1, z_1, y_2, z_2$ are distinct and appear in this order along the path $P_1 \cup P_2$. 
\end{enumerate}
We call edges in $P_1 \cup P_2$ \emph{supports} and the other edges \emph{rungs}.  Observe that our assumptions imply that the 
unique neighbour of $x_2$ on $P_1$ and the unique neighbour of $x_2$ on $P_2$ must be adjacent.  We say that $L$ 
is an \emph{extended} $(x_1,x_2,x_3)$-\emph{ladder} if $x_2$ is adjacent to a single vertex $x_2'$ in $L$ and $L \setminus x_2$ 
is a basic $(x_1,x_2',x_3)$-ladder.  We will say that $L$ is a $(x_1,x_2,x_3)$-\emph{ladder} if it is either a basic or extended 
$(x_1,x_2,x_3)$-ladder.

As we did with doublefans, it will be helpful to introduce a total preorder on the edges of a basic ladder.  To introduce this, let $L$
be a basic $(x_1,x_2,x_3)$-ladder as defined above and define a preorder on $E(L)$ by the following rule.  If $uv$ is a rung, then $uv \le f$
for every edge $f$ with both ends in the subpath of $P_1\cup P_2$ between $u$ and $v$.  If $uv$ is a support on the 
path $P_1$ with $u$ closer to $x_1$ than $v$ (along this path), then let $w$ be the furthest vertex from $v$ on the path $P_1 \cup P_2$ which is 
joined to $v$ by a rung.  Then we define $uv \le f$ for every edge $f$ with both ends in the subpath of $P_1 \cup P_2$ from $v$ to $w$.  Finally define the two edges of $P_1 \cup P_2$ incident with $x_2$ to be comparable with each other in both directions.
It is straightforward to check that this is indeed a total preorder.

\begin{lemma}
\label{df-lad}
For $i=0,1$ we do not have $\widetilde{G}_i$  a  $(v_1,v_2,v_3)$-doublefan with a delete-proof spine edge and $\widetilde{G}_{1-i}$ a $(v_1,v_2,v_3)$-ladder.  
\end{lemma}

\begin{center}
\begin{figure}[h]
\centerline{\includegraphics[scale=0.7]{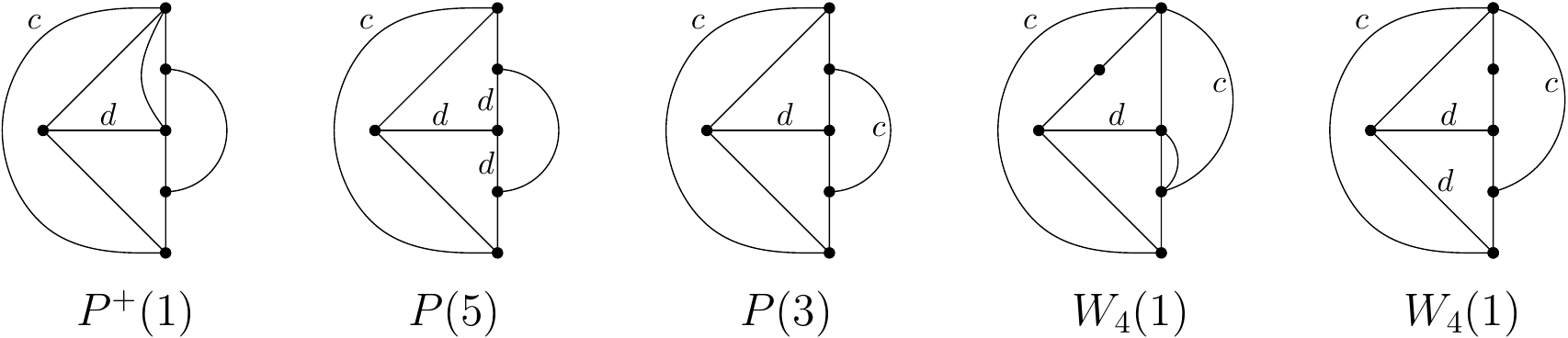}}
\caption{}
\label{df-lad1}
\end {figure}
\end{center}

\begin{proof} Suppose towards a contradiction that the lemma fails for $i=0$ (a similar argument works for $i=1$).  
Note that the delete-proof spine edge of $\widetilde{G}_0$ must be in $\widetilde{S}$ by Lemma \ref{mm3sep}, so we may assume (by possibly relabeling) that $e_2$ is a delete-proof spine edge.  If there is a rib edge which precedes $e_2$, then $\widetilde{G}$ contains $P^+(1)$ as a minor (as indicated in the first enhanced graph of Figure \ref{df-lad1}) which is a contradiction.  So, may assume that $e_2$ is the minimal edge in $\widetilde{G}_0$.  Similarly, by 3-connectivity, $\widetilde{G}$ will contain $P^+(1)$ if $\widetilde{G}_1$ is an extended ladder, so it must be that $\widetilde{G}_1$ is a basic ladder.  If the edges $\{e_4,e_5\}$ are the two elements incident with $v_2$ on this ladder, then $\widetilde{G}$ contains exactly three edges incident with $v_2$ and all are in $\widetilde{S}$, so all must be delete-proof.  It then follows that $\widetilde{G}$ contains a $P(5)$ minor (as indicated in the second enhanced graph of Figure \ref{df-lad1}) which is a contradiction.  So, we may assume that $e_4 < e_5$ without loss.  If $e_4$ is a rung, then $\widetilde{G} / e_4$ has a bad 2-separation, so $e_4$ must be contract-proof.  Conversely, if $e_4$ is a support, then $\widetilde{G} \setminus e_4$ has a bad 2-separation, so $e_4$ must be delete-proof.  We now split into cases depending on $e_4$.

\begin{center}
\begin{figure}[h]
\centerline{\includegraphics[height=4cm]{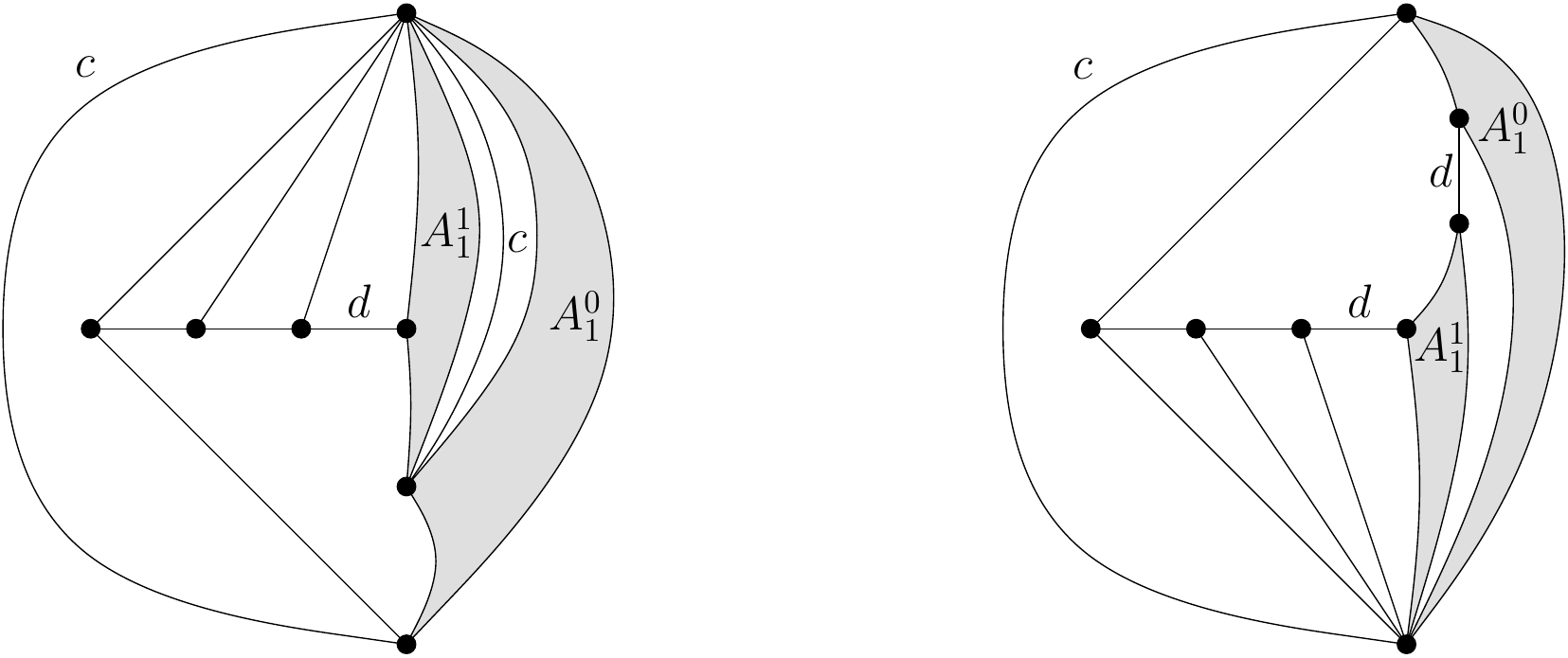}}
\caption{}
\label{df-lads}
\end {figure}
\end{center}

\startingcases
\begin{case} $e_4$ is a rung. \end{case}

If $e_4$ is not incident with either $v_1$ or $v_3$, then $\widetilde{G}$ contains a $P(3)$ minor as indicated in the third enhanced graph from Figure \ref{df-lad1} which 
is a contradiction.  So, by possibly reindexing, we may assume $e_4$ is incident with $v_1$.  If there is an inner $v_3$-rib of $\widetilde{G}_0$, then by contracting the edge of $P_1$ incident with $v_2$ we find that $\widetilde{G}$ contains a $W_4(1)$ minor (see the fourth enhanced graph in Figure \ref{df-lad1}) which is a contradiction. Therefore every inner rib of $G_0$ is incident with $v_1$.  Now setting $A^0_1 = \{ f \in A_1 \setminus \{e_4\} \mid f \le e_4 \}$ and $A_1^1 = A_1 \setminus (A_1^0 \cup \{e_4\})$ we find that $\widetilde{G}$ has the structure indicated on the left in Figure \ref{df-lads}.  If the edge $e_1$ is a spine edge, then it must be delete-proof, and $\widetilde{G}$ contains a $P^+(1)$ minor as before, which is contradictory.  So, $e_1$ must be a rib edge.  If $e_1$ is incident with $v_3$, then it must be delete proof and then $\widetilde{G}$ contains $W_4(1)$ as a minor (see the fifth enhanced graph in Figure \ref{df-lad1}), which is contradictory.  Therefore, $e_1$ is incident with $v_1$, and now it must be contract-proof.  However, again in this case $\widetilde{G}$ contains a $W_4(1)$ minor (this is similar to the fourth enhanced graph in Figure \ref{df-lad1} with the subdivided edge replaced by a contract-proof one), which is a contradiction.

\begin{center}
\begin{figure}[h]
\centerline{\includegraphics[height=3cm]{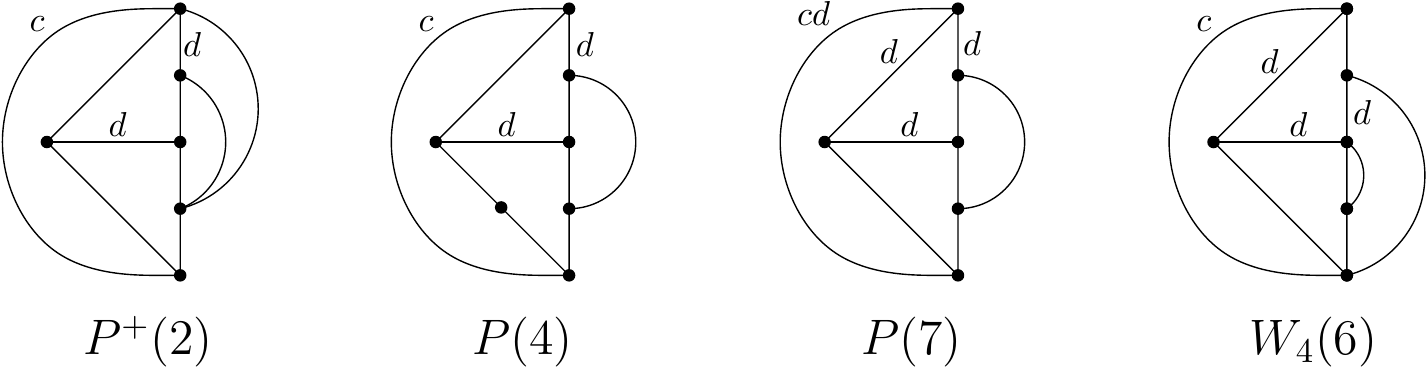}}
\caption{}
\label{df-lad2}
\end {figure}
\end{center}

\begin{case} $e_4$ is a support. \end{case}

We may assume (without loss) that $e_4$ is contained in the path $P_1$ from $v_1$ to $v_2$.  Let $Q$ be the component of $P_1 \setminus e_4$ 
containing $v_1$.  If there exists an edge between a vertex in $Q$ and the interior of the path $P_2$, then $\widetilde{G}$ contains a $P^+(2)$ minor 
(as shown in the first enhanced graph of Figure \ref{df-lad2}) which is contradictory.  Therefore, no such edge exists.  If there is an inner rib of $\widetilde{G}_0$ incident with  $v_1$, then $\widetilde{G}$ contains $P(4)$ (as shown in the second enhanced graph of Figure \ref{df-lad2}) which is a contradiction.  So, all inner ribs are incident with $v_3$.  Now setting $A_1^0 = \{ f \in A_1 \setminus \{e_4\} | f \le e_4 \}$ and $A_1^1 = A_1 \setminus (A_1^0 \cup \{e_4\})$ we find that $\widetilde{G}$ has the structure depicted on the right in Figure \ref{df-lads}.  Note that $A_1^0$ is empty if and only if $e_4$ is incident with $v_1$.  However $A_1^1$ contains the edge $e_5$, so it must be nonempty.  As before, we now turn our attention to the edge $e_1$.  If $e_1$ is a spine edge, then $e_1$ must be delete-proof and then $\widetilde{G}$ contains $P^+(1)$ as before, which is a contradiction.  If $e_1$ is incident with $v_3$, then $e_1$ must be contract-proof and we find that $\widetilde{G}$ contains a $P(4)$ minor (this is similar to the second enhanced graph in Figure \ref{df-lad2} except with the subdivided edge replaced by a contract-proof one), which is contradictory.  Therefore $e_1$ is the unique edge of $A_0$ incident with $v_1$.  In this case $e_1$ must be delete-proof.  If $A_1^0$ is empty, then all three edges incident with $v_1$ are in $\widetilde{S}$ and must therefore be delete-proof, and this gives us a $P(7)$ minor (as indicated in the third enhanced graph of Figure \ref{df-lad2}), which is a contradiction.  Otherwise, $A_1^0$ is nonempty and $\widetilde{G}$ contains a $W_4(6)$ minor as indicated in the fourth enhanced graph from Figure \ref{df-lad2}.  This final contradiction completes the proof.
\end{proof}

\subsection{Proof of theorem \ref{mainskeleton}}\label{subsec proof}

In this subsection we will complete our proof of the main result.  We begin by investigating the presence of small rooted minors.  For brevity, 
we will say that $\widetilde{G}_i$ contains one of the enhanced graphs $F$ from Figure \ref{triad_halves} if $(\widetilde{G}_i ; v_1,v_2,v_3)$ contains $(F; v_1, v_2, v_3)$ as a rooted minor.  

\begin{center}
\begin{figure}[h]
\centerline{\includegraphics[height=2.4cm]{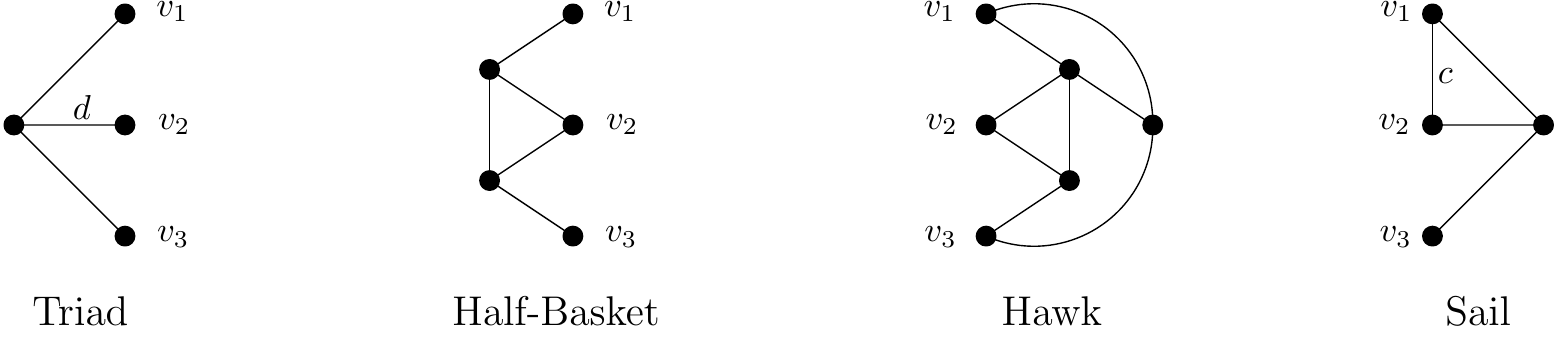}}
\caption{}
\label{triad_halves}
\end {figure}
\end{center}

\begin{lemma}
\label{no_triad} For $i=0,1$, either $\widetilde{G}_i$ contains triad, or it is a $(v_1,v_2,v_3)$-doublefan.
\end{lemma}

\begin{proof} 
If $\widetilde{G}_i \setminus \{v_1,v_3\}$ contains a cycle, there exist three vertex disjoint paths from $\{v_1,v_2,v_3\}$ to this cycle, and $\widetilde{G}_i$ contains triad.  Otherwise, Lemma \ref{basic_decomp} along with Lemmas \ref{iftriangle} and \ref{notriangle} imply that $\widetilde{G}_i$ is a $\{v_1,v_2,v_3\}$-doublefan, and the result follows easily.
\end{proof}

\begin{lemma}
\label{triad_struc}
For $i=0,1$, if $\widetilde{G}_i$ contains triad, then $\widetilde{G}_{1-i}$ is either a $(v_1,v_2,v_3)$-doublefan or a $(v_1,v_2,v_3)$-ladder. 
\end{lemma}

\begin{proof}
We will not use the size of the $A_i$ in this argument, so we may assume without loss that $i=0$. If $v_2$ is incident with at least two edges in $A_1$ then set $G' = \widetilde{G}_1$ and
set $v_2' = v_2$.  Otherwise, let $v_2'$ be the vertex adjacent to $v_2$ by an edge in $A_1$ and set $G' = \widetilde{G}_1 \setminus v_2$.  

\begin{center}
\begin{figure}[h]
\centerline{\includegraphics[scale=0.65]{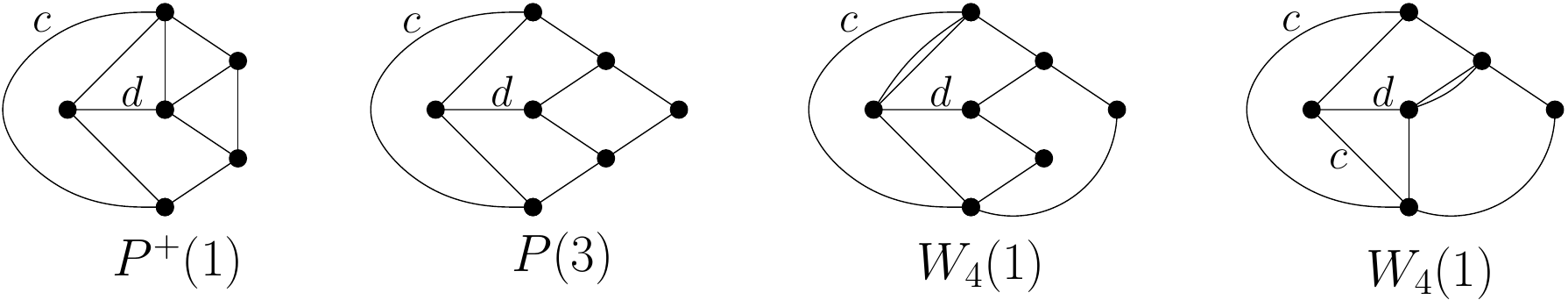}}
\caption{}
\label{triad_extras}
\end {figure}
\end{center}

First suppose that $v_2'$ is adjacent to one of $v_1,v_3$.  If the enhanced graph $G' \setminus \{v_1,v_3\}$ contains a cycle, then 
$\widetilde{G}$ contains $P^+(1)$ as a minor (as seen in the first enhanced graph in Figure \ref{triad_extras}) which is contradictory.  Otherwise, $\widetilde{G}_1$ is a
$(v_1,v_2,v_3)$-doublefan and we are finished.  So, we may assume $v_2'$ is not adjacent to $v_1,v_3$.
Now, $\widetilde{G}$ is a 3-connected planar graph, and $v_1,v_2',v_3$ form a 3 vertex separation, so there is a path $P_1 \subseteq G'$
from $v_1$ to $v_2'$ which forms a part of a facial cycle, and there is a similar path $P_2 \subseteq G'$ from $v_2'$ to $v_3$.  Note that our assumptions imply that these paths both have length at least two, and meet only at $v_2'$.  

There must not exist a path of length at least two which is internally disjoint from $P_1 \cup P_2$ and has one end in the interior of $P_1$
and the other in the interior of $P_2$, as in this case $\widetilde{G}$ contains $P(3)$ (as shown in the second graph from Figure \ref{triad_extras}).
Since the graph obtained from the union of the faces of $\widetilde{G}$ containing $v_2'$ by deleting the vertex $v_2'$ is a cycle 
(by 3-connectivity and planarity) it then follows that $v_2'$ is adjacent with exactly two vertices in $G'$, one vertex 
$v_2^-$ in $V(P_1)$ and one vertex $v_2^+$ in $V(P_2)$ and these three vertices form a triangular face.  

If $V(G') = V(P_1) \cup V(P_2)$ then the planarity and 3-connectivity of $\widetilde{G}$ imply that $G'$ is a 
$(v_1,v_2',v_3)$-ladder and we are finished.  Otherwise, let $X$ be the vertex set of a component of $G' \setminus (V(P_1) \cup V(P_2))$. 
There must not exist a vertex in the interior of $P_1$ adjacent to a point in $X$ and another in the interior of $P_2$ adjacent to a point in $X$, as then 
$\widetilde{G}$ would contain $P(3)$ (as before).  Thus by 3-connectivity and planarity (and symmetry between $v_1$ and $v_3$), we may assume that there is a vertex in $X$ adjacent to $v_3$ and all other vertices outside $X$ which are adjacent to a point in $X$ lie in the path $P_1$.    Note that this implies 
that $G'$ contains Hawk.

\begin{center}
\begin{figure}[h]
\centerline{\includegraphics[height=4.3cm]{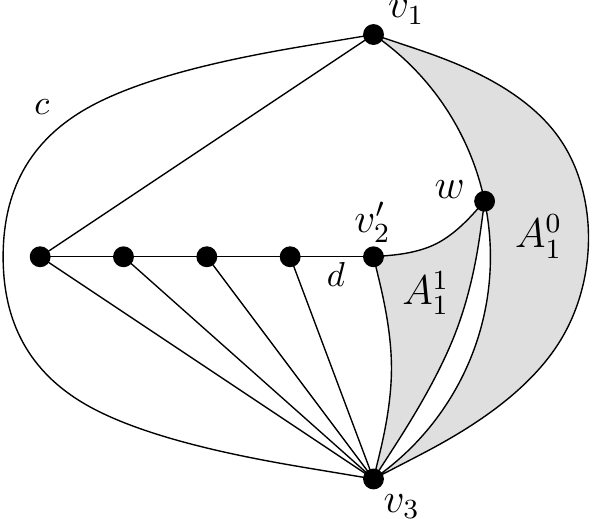}}
\caption{}
\label{triad_lem_struc}
\end {figure}
\end{center}

Let $w$ be the vertex on the path $P_1$ nearest $v_2'$ which is adjacent to a point in $X$.  Now, by planarity, there exists a partition 
of $E(G')$ into $\{A_1^0, A_1^1 \}$ so that $\partial(A_1^0) = \{v_1, w, v_3\}$ and $\partial(A_1^1) = \{ v_2', w, v_3 \}$ as on the right hand side of Figure \ref{triad_lem_struc}.  Furthermore, $|A_1^0|, |A_1^1| \ge 4$, so by Lemma \ref{mm3sep} we have $|A_1^0 \cap S| = |A_1^1 \cap S| = 1$.  

Equipped with this information, we now turn our attention to $\widetilde{G}_0$.  If $\widetilde{G}_0 \setminus \{v_1,v_3\}$ contains
a cycle, then $\widetilde{G}_0$ contains Half-Basket, and since $\widetilde{G}_1$ also contains Half-Basket, we find that $\widetilde{G}$ contains $D^*(1)$, which is a contradiction.  It follows that $\widetilde{G}_0$ is a $(v_1,v_2,v_3)$-doublefan.  Let $u_1, u_2, \ldots, u_n$ be the vertex sequence in the path $\widetilde{G}_0 \setminus \{v_1,v_3\}$ and assume $v_2 = u_1$.  Note that the assumption $\widetilde{G}_0$ contains triad implies that $\widetilde{G}_0$ has a delete-proof spine edge $u_i u_{i+1}$.  If there is a smaller rib edge, then $\widetilde{G}$ contains $P^+(1)$ (as shown in the first graph of Figure \ref{triad_extras}) which is a contradiction.  So, by 3-connectivity the only delete-proof spine edge is $u_1 u_2 = v_2 u_2$.  For the same reason, it must be that $v_2' = v_2$.  If there is an inner $v_1$-rib, then $\widetilde{G}$ contains $W_4(1)$ (as shown in the third graph of Figure \ref{triad_extras}) which is a contradiction.  So, we may assume that every inner rib is incident with $v_3$, and Figure \ref{triad_lem_struc} depicts our graph.  

The edge $u_1 u_2$ must be in $\widetilde{S}$ by Lemma \ref{mm3sep}, so we may assume $e_2 = u_1 u_2$.  If $e_1$ is a spine edge, then it must be delete-protected, (since $\widetilde{G} \setminus e_1$ has a bad 2-separation using the vertices $w,v_3$).  However, then $\widetilde{G}$ contains $P^+(1)$ (as shown in the first graph of Figure \ref{triad_extras}) which is contradictory.  If $e_1$ is a $v_3$-rib then it must be contract proof, and $\widetilde{G}$ contains $W_4(1)$ (as shown in the last graph of Figure \ref{triad_extras}) which is contradictory.  Thus $e_1 = v_1 u_n$, and this edge must be  delete-proof.  However then $\widetilde{G}$ has a $K_5^-(4)$ minor similar to that in the third graph of Figure \ref{triad_extras}.  This final contradiction completes the proof.
\end{proof}

Now we are ready to complete our results by proving Lemma \ref{fminor}.

\bigskip

\noindent{\it Proof of Lemma \ref{fminor}.} If neither $\widetilde{G}_0$ nor $\widetilde{G}_1$ contains triad, then Lemma \ref{no_triad} implies that both are $(v_1,v_2,v_3)$-doublefans contradicting Lemma \ref{central_doublefan}.  If they both contain triad, then by Lemma \ref{triad_struc} each of $\widetilde{G}_0,\widetilde{G}_1$ must either be a $(v_1,v_2,v_3)$-doublefan or a $(v_1,v_2,v_3)$-ladder.  If both are $(v_1,v_2,v_3)$-doublefans, then we again have a contradiction to Lemma \ref{central_doublefan}.  If both are $(v_1,v_2,v_3)$-ladders, then each contains Half-Basket, so $\widetilde{G}$ has $D^*(1)$ as a minor, which is a contradiction.  Finally, if exactly one is a ladder, then the other is a $(v_1,v_2,v_3)$-doublefan with a delete-proof spine edge, contradicting Lemma \ref{df-lad}.  

In the remaining case we may assume $\widetilde{G}_0$ does not contain triad but $\widetilde{G}_1$ does.  It follows from Lemma \ref{no_triad} and the assumption that $\widetilde{G}_0$ has no triad that $\widetilde{G}_0$ is a $(v_1,v_2,v_3)$-doublefan without a delete-proof spine edge.  Since there is no triangle of edges in $\widetilde{S}$, we may then assume (by possibly switching $v_1$ and $v_3$) that $\widetilde{G}_0$ contains Sail.  If $\widetilde{G}_1 \setminus \{v_1,v_2,v_3\}$ contains a cycle, then $\widetilde{G}$ has a $P(1)$ minor which is contradictory.  Thus, Lemma \ref{basic_decomp} implies that $\widetilde{G}_1$ is a $\{v_1,v_2,v_3\}$-doublefan, and now Lemma \ref{general_doublefan} gives us a contradiction. \qed

\section{Code}

Throughout this project, we used computers in a variety of ways.  The lion's share of code was written in Sage \cite{sagemath}, a freely available, open source mathematical software package built on top of the Python programming language.  Initially, we developed code to explore the splitting property of graphs ``from the bottom up,'' considering for each configuration of five edges the possible ways in which their endpoints were connected within a larger graph.  The code looked for graphs that had induced spanning trees in minors corresponding to the various five configurations, using Proposition \ref{calcdodgson} in Section \ref{sec splitting}.  This approach yielded limited results for certain connected configurations, but the number of possibilities quickly became unwieldy for more general configurations.

We then took the opposite ``top down'' approach to characterizing split graphs, starting with the graph and looking at various configurations of five edges contained in it.  This latter approach proved to be much more effective, although not initially.  The computational overhead for calculating Kirchhoff matrices and Dodgson polynomials for each combination of five edges was simply too high (more than a week running on 4 parallel 2.13 GHz cores to run through graphs up to 12 edges).  This bottleneck was somewhat alleviated by using large random integers rather than symbolic polynomials, although this opened the door to the (extremely unlikely) possibility of falsely determining a split five configuration as a result of a numerical coincidence.  It should be stressed that these calculations were meant as rough guides towards the characterization of the set of forbidden minors, not as the substance of rigorous proofs.

Each new discovery, as a result of an exhaustive search in Sage, led to new theoretical developments that necessitated rewriting the code to optimize the search.  Especially after relaxing the 3-connected hypothesis, the pantheon of minor minimal non-splitting graphs grew with each larger size of graphs tested.  The search was reaching the limits of feasibility, using brute-force methods.  Here the development of the matroid approach, using enhanced graphs (Section \ref{wands}) was pivotal.  The calculation was broken into ``phases,'' initially assuming that all types of protection were both contract- and delete-proof (one of the configurations of three edges on the right of Figure \ref{ordinary-enhanced}).  Those graphs that did not split were passed to the second phase, in which the smaller choices of protection were tested.  In this way, \emph{minimal} non-splitting graphs were found.  To alleviate the significant memory strain, results from the first phase were written to file in a systematic way and retrieved for the second phase of the calculation.  All told, the total processing time for this final successful approach measured in days.  Using this approach we exhaustively checked all enhanced graphs where the underlying graph had at most 11 edges.

Happily, this exhaustive search for minimal non-split graphs agrees completely with the characterization in Theorem \ref{mainskeleton}.

\appendix 

\section{All excluded minors}\label{app all_excl}

The full list of excluded minors for splitting is $K_5$, $K_{3,3}$, the octahedron $O$, the cube $C$, 
$H$ as defined in Figure \ref{firsth},
and the following enhanced graphs:

\noindent
    \includegraphics[width=\linewidth]{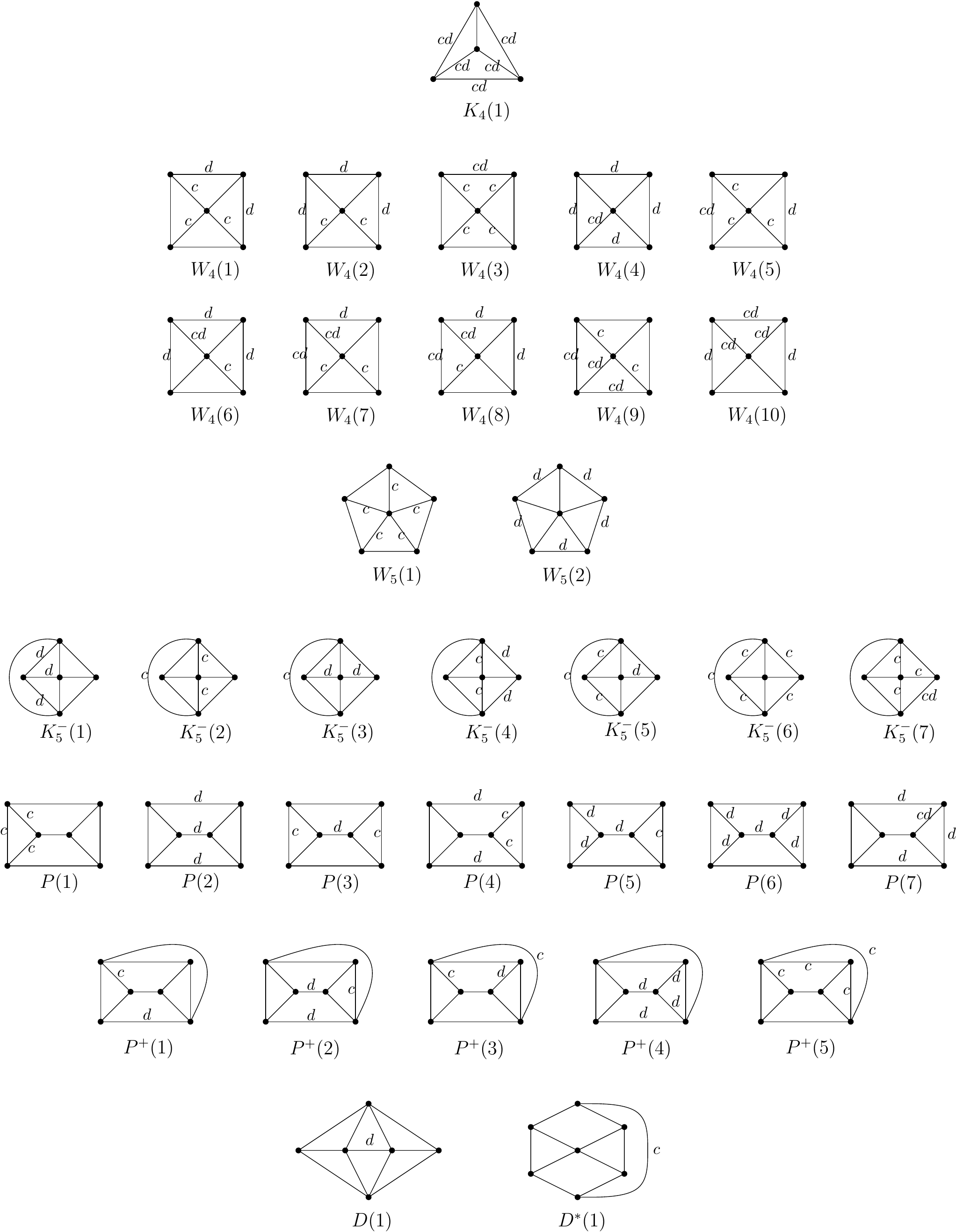}

\newpage

\bibliographystyle{plain}
\bibliography{split}

\begin{thebibliography}{10}

\bibitem{AMdet}
Paolo Aluffi and Matilde Marcolli.
\newblock Parametric {F}eynman integrals and determinant hypersurfaces.
\newblock {\em Adv. Theor. Math. Phys.}, 14(3):911--964, 2010.
\newblock arXiv:0901.2107.

\bibitem{bek}
Spencer Bloch, H\'el\`ene Esnault, and Dirk Kreimer.
\newblock On motives associated to graph polynomials.
\newblock {\em Commun. Math. Phys.}, 267:181--225, 2006.
\newblock arXiv:math/0510011v1 [math.AG].

\bibitem{bkphi4}
D.J. Broadhurst and D.~Kreimer.
\newblock Knots and numbers in $\phi^4$ theory to 7 loops and beyond.
\newblock {\em Int.J.Mod.Phys.}, C6(519-524), 1995.
\newblock arXiv:hep-ph/9504352.

\bibitem{Brbig}
Francis Brown.
\newblock On the periods of some {F}eynman integrals.
\newblock arXiv:0910.0114.

\bibitem{BrS}
Francis Brown and Oliver Schnetz.
\newblock A {K3} in $\phi^4$.
\newblock {\em Duke Math J.}, 161(10):1817--1862, 2012.
\newblock arXiv:1006.4064.

\bibitem{BrSY}
Francis Brown, Oliver Schnetz, and Karen Yeats.
\newblock Properties of $c_2$ invariants of {F}eynman graphs.
\newblock arXiv:1203.0188.

\bibitem{BrY}
Francis Brown and Karen Yeats.
\newblock Spanning forest polynomials and the transcendental weight of
  {F}eynman graphs.
\newblock {\em Commun. Math. Phys.}, 301(2):357--382, 2011.
\newblock arXiv:0910.5429.

\bibitem{MR3090713}
Guoli Ding.
\newblock A characterization of graphs with no octahedron minor.
\newblock {\em J. Graph Theory}, 74(2):143--162, 2013.

\bibitem{DDpathwidth}
Guoli Ding and Stan Dziobiak.
\newblock 3-connected graphs of path-width at most three.
\newblock preprint.

\bibitem{HJ}
R.~Halin and H.A. Jung.
\newblock Uber minimalstrukturen von graphen, insbesondere von n--fach
  zusammenhangenden graphen.
\newblock {\em Math. Ann.}, 152:75--94, 1963.

\bibitem{numbers}
Dirk Kreimer.
\newblock The residues of quantum field theory---numbers we should know.
\newblock In {\em Noncommutative geometry and number theory}, Aspects Math.,
  E37, pages 187--203. Vieweg, Wiesbaden, 2006.

\bibitem{Mcube}
John Maharry.
\newblock A characterization of graphs with no cube minor.
\newblock {\em Journal of Combinatorial Theory, Series B}, 80:179--201, 2000.

\bibitem{PScube}
Themistocles Politof and A.~Satyanarayana.
\newblock A linear-time algorithm to compute the reliability of planar
  cube-free networks.
\newblock {\em IEEE Transactions on Reliability}, 39:557--563, 1990.

\bibitem{RSX}
Neil Robertson and P.D. Seymour.
\newblock Graph minors. {X}. obstructions to tree-decomposition.
\newblock {\em Journal of Combinatorial Theory, Series B}, 52(2):153--190,
  1991.

\bibitem{RSXX}
Neil Robertson and P.D. Seymour.
\newblock Graph minors. {XX}. {W}agner's conjecture.
\newblock {\em Journal of Combinatorial Theory, Series B}, 92(2):325--357,
  2004.

\bibitem{sagemath}
Sage mathematics software system.
\newblock \url{http://www.sagemath.org/}.
\newblock Accessed: 2013-09-07.

\bibitem{Sphi4}
Oliver Schnetz.
\newblock Quantum periods: A census of $\phi^4$-transcendentals.
\newblock {\em Communications in Number Theory and Physics}, 4(1):1--48, 2010.
\newblock arXiv:0801.2856.

\bibitem{SFq}
Oliver Schnetz.
\newblock Quantum field theory over $\mathbb{F}_q$.
\newblock {\em Electronic Journal of Combinatorics}, 18(1):P102, 2011.
\newblock arXiv:0909.0905.

\bibitem{Twidth}
Dimitrios~M. Thilikos.
\newblock Algorithms and obstructions for linear-width and related search
  parameters.
\newblock {\em Discrete Applied Mathematics}, 105:239--271, 2000.

\bibitem{Ycov}
Karen Yeats.
\newblock Some combinatorial interpretations in perturbative quantum field
  theory.
\newblock arXiv:1302.0080.

\bibitem{YuYDY}
Yaming Yu.
\newblock More forbidden minors for wye-delta-wye reducibility.
\newblock {\em Electronic Journal of Combinatorics}, 13:R7, 2006.

\end{thebibliography}

\end{document}